\documentclass[numbers,webpdf,imanum]{ima-authoring-template}%
\usepackage{stmaryrd}
\usepackage{float}
\usepackage{microtype}
\usepackage{upgreek}
\usepackage{mathtools}
\usepackage{paper_diening}
\usepackage{mathrsfs}
\usepackage{hyperref}
\usepackage{booktabs}
\usepackage{colortbl}
\usepackage{tabulary}
\usepackage{diagbox}
\usepackage{caption}
\usepackage{esint}

\usepackage{calrsfs}
\DeclareMathAlphabet{\pazocal}{OMS}{zplm}{m}{n}
\DeclareMathAlphabet{\pazobfcal}{OMS}{zplm}{b}{n}

\graphicspath{{Fig/}}

\theoremstyle{thmstyletwo}%
\newtheorem{theorem}{Theorem}
\newtheorem{proposition}[theorem]{Proposition}%
\newtheorem{lemma}[theorem]{Lemma}%
\newtheorem{remark}[theorem]{Remark}%
\newtheorem{corollary}[theorem]{Corollary}%

\newtheorem{assumption}[theorem]{Assumption}

\numberwithin{equation}{section}

\providecommand{\meantmp}[2]{#1\langle{#2}#1\rangle}
\providecommand{\mean}[1]{\meantmp{}{#1}}

\providecommand{\jumptmp}[2]{#1\llbracket{#2}#1\rrbracket}
\providecommand{\jump}[1]{\jumptmp{}{#1}}

\providecommand{\avgtmp}[2]{#1\{{#2}#1\}}
\providecommand{\avg}[1]{\avgtmp{}{#1}}

\providecommand{\PiDG}{{\Uppi_{h}^{k}}}

\providecommand{\Pia}{{\Uppi_h^{0}}}

\providecommand{\Vo}{{\mathaccent23 V}}
\providecommand{\Qo}{{\mathaccent23 Q}}
\providecommand{\Xhk}{\smash{X_h^k}}
\providecommand{\Vhk}{\smash{V_h^k}}
\providecommand{\Qhk}{\smash{Q_h^k}}
\providecommand{\Qhkc}{\smash{Q_{h,c}^k}}

\providecommand{\Qhkco}{{{\mathaccent23 Q}_{h,c}^k}}

\providecommand{\SSS}{\mathbf{S}}

\newcommand{\Ghk}{\boldsymbol{\pazobfcal{G}}_h^k}

\newcommand{\Dhk}{\boldsymbol{\pazobfcal{D}}_h^k}

\newcommand{\Divhk}{\pazobfcal{D}\hspace*{-0.05em}\dot{\iota}\!\nu_h^k}

\newcommand{\setRhk}{\boldsymbol{\pazobfcal{R}}_h^k}

\newcommand{\WDG}{W^{1,p}(\pazocal{T}_h)}

\providecommand{\divo}{\mathrm{div}\,}

\providecommand{\sssl}{\avg{\abs{\Pia \Dhk \bfv_h}}}

\newcommand\blfootnote[1]{%
	\begingroup
	\renewcommand\thefootnote{}\footnote{#1}%
	\addtocounter{footnote}{-1}%
	\endgroup
}

\begin{document}

\DOI{DOI HERE}
\copyrightyear{2021}
\vol{00}
\pubyear{2021}
\access{Advance Access Publication Date: Day Month Year}
\appnotes{Paper}
\copyrightstatement{Published by Oxford University Press on behalf of the Institute of Mathematics and its Applications. All rights reserved.}
\firstpage{1}


\title[Note on quasi-optimal error estimates for the pressure for shear-thickening fluids]{Note on quasi-optimal error estimates for the pressure for shear-thickening fluids}

\author{Alex Kaltenbach*\ORCID{0000-0001-6478-7963}
\address{\orgdiv{Institute of Mathematics}, \orgname{Technical University of Berlin},\\ \orgaddress{\street{Straße des 17.~Juni 135}, \postcode{10623}, \state{Berlin}, \country{Germany}}}}
\corresp[*]{Alex Kaltenbach: \href{email:kaltenbach@math.tu-berlin.de}{kaltenbach@math.tu-berlin.de}}
\author{Michael R\r{u}\v{z}i\v{c}ka
\address{\orgdiv{Department of Applied Mathematics}, \orgname{University of Freiburg},\\  \orgaddress{\street{Ernst--Zermelo-Straße 1}, \postcode{79104}, \state{Freiburg}, \country{Germany}}}}

\authormark{Alex Kaltenbach and Michael R\r{u}\v{z}i\v{c}ka}

\received{Date}{0}{Year}
\revised{Date}{0}{Year}
\accepted{Date}{0}{Year}


\abstract{\hspace*{5mm}In this paper, we derive quasi-optimal a priori error estimates for the kinematic pressure~for~a~Local Discontinuous Galerkin (LDG) approximation  of steady systems of $p$-Navier--Stokes type in the case of shear-thickening, \textit{i.e.}, in the case $p>2$, imposing a new mild Muckenhoupt regularity condition.
}
 
\maketitle

\section{Introduction}\blfootnote{* funded by the Deutsche Forschungsgemeinschaft (DFG, German Research Foundation) - 525389262\vspace*{-1mm}}

\hspace*{5mm}In the present paper, we examine a \textit{Local Discontinuous Galerkin (LDG)} approximation of 
steady systems of \textit{$p$-Navier--Stokes  type}, \textit{i.e.}, 
\begin{equation}
	\label{eq:p-navier-stokes}
	\begin{aligned}
		-\divo\SSS(\bfD\bfv)+[\nabla\bfv]\bfv+\nabla q&=\bff   \qquad&&\text{in }\Omega\,,\\
		\divo\bfv&=0 \qquad&&\text{in }\Omega\,,
		\\
		\bfv &= \mathbf{0} &&\text{on } \partial\Omega\,.
	\end{aligned}
\end{equation}
for quasi-optimal a priori error estimates for the kinematic pressure  in the case of shear-thickening~fluids.
The system \eqref{eq:p-navier-stokes} describes the steady motion of a homogeneous,
incompressible fluid with shear-dependent viscosity. More precisely,
for a given vector field $\bff \colon\Omega\to \setR^d$ describing external~forces,~an incompressibility constraint  \eqref{eq:p-navier-stokes}$_2$,  and a non-slip
boundary condition \eqref{eq:p-navier-stokes}$_3$, the system \eqref{eq:p-navier-stokes} seeks for a
\textit{velocity vector field} $\bfv=(v_1,\ldots,v_d)^\top\colon \Omega\to
	\setR^d $ and a \textit{kinematic pressure} $q\colon \Omega\to \setR$ solving~\eqref{eq:p-navier-stokes}.
Here, $\Omega\subseteq \mathbb{R}^d$, $d\in \set{2,3}$, is a bounded, polygonal (if $d=2$)~or~polyhedral~(if~${d=3}$)~Lipschitz domain. The \textit{extra stress tensor} $\SSS(\bfD\bfv)\colon\Omega\to \smash{\setR^{d\times d}_{\textup{sym}}}$ depends on the \textit{strain rate tensor} $\bfD\bfv\coloneqq \frac{1}{2}(\nabla\bfv+\nabla\bfv^\top)\colon\Omega\to \smash{\setR^{d\times d}_{\textup{sym}}}$, \textit{i.e.}, the symmetric part of the velocity gradient  $\nabla\bfv=(\partial_j v_i)_{i,j=1,\ldots,d}
\colon\Omega\to \setR^{d\times d}$. The \textit{convective term} $\smash{[\nabla\bfv]\bfv\colon\Omega\to \mathbb{R}^d}$~is~defined~via $\smash{([\nabla\bfv]\bfv)_i\coloneqq \sum_{j=1}^d{v_j\partial_j v_i}}$ for all $i=1,\ldots,d$.

Throughout the paper, we  assume that the extra stress tensor~$\SSS$~has~\textit{$(p,\delta)$-structure} (\textit{cf.}~Assumption~\ref{assum:extra_stress}). The relevant example falling into this class is 
\begin{align*}
	\smash{\SSS(\bfD\bfv)=\mu_0\, (\delta+\vert \bfD\bfv\vert)^{p-2}\bfD\bfv}\,,
\end{align*}
where $p\in (1,\infty)$, $\delta\ge 0$, and $\mu_0>0$.

\hspace*{-1mm}The a priori error analysis of the steady $p$-Navier--Stokes
problem \eqref{eq:p-navier-stokes} using FE~or~DG~approximations is by now well-understood (see \cite{kr-pnse-ldg-1} for an overview): 
recently, 
 in \cite{kr-pnse-ldg-2,kr-pnse-ldg-3}, a priori error estimates for an LDG approximation of the steady $p$-Navier--Stokes
problem \eqref{eq:p-navier-stokes} in the case of shear-thickening, \textit{i.e.}, in the case $p>2$, were derived, which are optimal for the velocity vector field, but sub-optimal for the kinematic pressure.
This lacuna is mainly due to the following technical hurdle:
in the error analysis of the  LDG approximations of the steady $p$-Navier--Stokes problem \eqref{eq:p-navier-stokes}, it turn out that the error of the kinematic pressure measured in the $L^{p'}(\Omega)$-norm is bounded by the $L^{(\varphi_{\smash{\vert \bfD\bfv\vert}})^*}(\Omega)$-norm of the~stress~errors, \textit{i.e.},
\begin{align*}
	\smash{\| q_h-q\|_{p',\Omega}\leq \|\bfS(\Dhk\bfv_h)-\bfS(\bfD\bfv)\|_{(\varphi_{\smash{\vert \bfD\bfv\vert}})^*,\Omega}+h\,\|\bfS_{\smash{\avg{\vert \Pi_h^0\Dhk\bfv_h\vert }}}(h^{-1}\jump{\bfv_h\otimes\bfn})\|_{(\varphi_{\smash{\avg{\vert \Pi_h^0\Dhk\bfv_h\vert }}})^*,\Gamma_h}+\textup{(h.o.t.)}\,.}
\end{align*}
This relation is mainly a consequence of the discrete inf-sup stability result (\textit{cf.}\ \cite[Lem.\ 6.10]{DiPE12}) 
\begin{align*}
	\| z_h\|_{p',\Omega}\leq \sup_{\bfz_h\in  V_h^k\,:\,\|\bfz_h\|_{\nabla,p,h}\leq 1}{(z_h,\Divhk \bfz_h)_{\Omega}}\,.\\[-7mm]
\end{align*} 
Then, using the estimates
\begin{align*}
	\|\bfS(\Dhk\bfv_h)-\bfS(\bfD\bfv)\|_{(\varphi_{\smash{\vert \bfD\bfv\vert}})^*,\Omega}^2&\leq 
	c\,\|\bfF(\Dhk\bfv_h)-\bfF(\bfD\bfv)\|_{2,\Omega}^2\,,\\[-1mm]
h\,	\|\bfS_{\smash{\avg{\vert \Pi_h^0\Dhk\bfv_h\vert }}}(h^{-1}\jump{\bfv_h\otimes\bfn})\|_{(\varphi_{\smash{\avg{\vert \Pi_h^0\Dhk\bfv_h\vert }}})^*,\Gamma_h}&\leq 
c\,h\,\rho_{\varphi_{\smash{\avg{\vert \Pi_h^0\Dhk\bfv_h\vert }}},\Gamma_h}(h^{-1}\jump{\bfv_h\otimes\bfn})\,.
\end{align*}
one finds that\vspace*{-0.5mm}
\begin{align*}
	\smash{\| q_h-q\|_{p',\Omega}^2\leq c\,\big(\|\bfF(\Dhk\bfv_h)-\bfF(\bfD\bfv)\|_{2,\Omega}^2+h\,\rho_{\varphi_{\smash{\avg{\vert \Pi_h^0\Dhk\bfv_h\vert }}},\Gamma_h}(h^{-1}\jump{\bfv_h\otimes\bfn})\big)+\textup{(h.o.t.)}\,,}
\end{align*}
\textit{i.e.}, the squared kinematic pressure error measured in the $L^{p'}(\Omega)$-norm is bounded by the~squared~velocity vector field error.
However, numerical experiments (\textit{cf.}\ \cite{kr-pnse-ldg-3}) suggest the relation 
\begin{align}
	\label{intro:relation}
	\rho_{(\varphi_{\vert \bfD\bfv\vert})^*,\Omega}(q_h-q)\leq c\,\big(\|\bfF(\Dhk\bfv_h)-\bfF(\bfD\bfv)\|_{2,\Omega}^2+h\,\rho_{\varphi_{\smash{\avg{\vert \Pi_h^0\Dhk\bfv_h\vert }}},\Gamma_h}(h^{-1}\jump{\bfv_h\otimes\bfn})\big)+\textup{(h.o.t.)}\,.
\end{align}
To be in the position to establish the relation \eqref{intro:relation}, however, it is necessary to have a discrete convex conjugation inequality in terms of the shifted modular $	\rho_{\varphi_{\vert \bfD\bfv\vert},\Omega}$, \textit{i.e.},
\begin{align}\label{intro:discrete_convex_conjugation_ineq}
	\rho_{(\varphi_{\vert \bfD\bfv\vert})^*,\Omega}(z_h)\leq \sup_{\bfz_h\in  V_h^k}{\big[(z_h,\Divhk\,\bfz_h)_{\Omega}-\tfrac{1}{c}\big(\rho_{\varphi_{\vert \bfD\bfv\vert},\Omega}(\nabla\bfz_h)+\rho_{\varphi_{\vert \bfD\bfv\vert},\Gamma_h}( \jump{\bfz_h\otimes \mathbf{n}})\big)\big]}\,.\\[-7mm]\notag
\end{align}
The validity of the discrete convex conjugation inequality \eqref{intro:discrete_convex_conjugation_ineq}, in turn, requires the validity~of~a~contin-uous analogue and the stability of the Bogovski\u{\i} operator with respect to the shifted modular~$	\rho_{\varphi_{\vert \bfD\bfv\vert},\Omega}$. \linebreak In this paper, we will establish that the latter is available under the additional assumption that the viscosity of the fluid is a Muckenhoupt weight of class $2$, \textit{i.e.}, if we have that\vspace*{-0.5mm}
\begin{align}\label{intro:Muckenhoupt_regularity}
	\smash{\mu_{\bfD\bfv} \coloneqq (\delta+\vert \overline{\bfD\bfv}\vert)^{p-2}\in A_2(\mathbb{R}^d)\,,}
\end{align}
where $\overline{\bfD\bfv}\coloneqq \bfD\bfv$ a.e.\ in $\Omega$ and $\overline{\bfD\bfv}\coloneqq \bfzero$ a.e.\  in $\mathbb{R}^d\setminus\Omega$. In \cite{kr-nekorn,kr-nekorn-add}, it turned out that the Muckenhoupt regularity 
assumption \eqref{intro:Muckenhoupt_regularity} can not be expected, in general, in the three dimensional case. However, in the two dimensional case, regularity results (\textit{cf.}\ \cite{KMS2}) suggest that \eqref{intro:Muckenhoupt_regularity} is satisfied under mild assumptions.\enlargethispage{5mm}

\textit{This paper is organized as follows:} 
In Section
\ref{sec:preliminaries}, we introduce the employed~\mbox{notation},~define~relevant function spaces, 
basic assumptions on the extra stress~tensor~$\SSS$~and~its consequences, the weak formulations Problem (\hyperlink{Q}{Q}) and
Problem~(\hyperlink{P}{P})~of~the~system~\eqref{eq:p-navier-stokes}, and the  discrete
operators.~In~Section~\ref{sec:ldg}, we introduce the discrete weak
formulations Problem (\hyperlink{Qhldg}{Q$_h$}) and Problem
(\hyperlink{Phldg}{P$_h$}), recall known a priori error estimates and
prove the main result of the paper, \textit{i.e.}, a quasi-optimal (with
respect to the Muckenhoupt regularity condition
\eqref{intro:Muckenhoupt_regularity}) a priori error estimate for the
kinematic pressure in the case $p>2$ (\textit{cf.}~\hspace*{-0.1mm}Theorem~\hspace*{-0.1mm}\ref{thm:error_pressure}, \hspace*{-0.1mm}Corollary \hspace*{-0.1mm}\ref{cor:error_pressure}).
 \hspace*{-0.1mm}In \hspace*{-0.1mm}Section \hspace*{-0.1mm}\ref{sec:experiments}, \hspace*{-0.1mm}we \hspace*{-0.1mm}review~\hspace*{-0.1mm}the~\hspace*{-0.1mm}theoretical~\hspace*{-0.1mm}findings~\hspace*{-0.1mm}via~\hspace*{-0.1mm}numerical~\hspace*{-0.1mm}experiments.

\newpage
\section{Preliminaries}\label{sec:preliminaries}

\subsection{Basic Notation}

\hspace*{5mm}We use the notation of the papers \cite{kr-pnse-ldg-1,kr-pnse-ldg-2,kr-pnse-ldg-3}. For the convenience of the reader, we repeat~some~of~it.\enlargethispage{2mm}

We employ $c,C>0$ to denote generic constants,~that~may change from line to line, but do not depend on the crucial quantities. Moreover, we~write~$u\sim v$ if and only if there exist constants $c,C>0$ such that $c\, u \leq v\leq C\, u$.

For $k\in \setN$ and $p\in [1,\infty]$, we employ the customary
Lebesgue spaces $(L^p(\Omega), \|\cdot\|_{p,\Omega}) $ and Sobolev
spaces $(W^{k,p}(\Omega),\|\cdot\|_{k,p,\Omega})$, where $\Omega
\subseteq \setR^d$, $d\in \{2,3\}$, is a bounded, polygonal (if $d=2$) or polyhedral (if $d=3$) Lipschitz domain. The space $\smash{W^{1,p}_0(\Omega)}$
is defined as those functions from $W^{1,p}(\Omega)$ whose traces vanish on $\partial\Omega$. 
The Hölder dual exponent is denoted by $p' = \tfrac{p}{p-1} \in [1,
\infty]$. 

We denote vector-valued functions by boldface letters~and~tensor-valued
functions by capital boldface letters.
The Euclidean inner product
between two vectors $\bfa =(a_1,\ldots,a_d)^\top,\bfb=(b_1,\ldots,b_d)^\top\in \mathbb{R}^d$ is denoted by 
$\bfa \cdot\bfb\coloneqq \sum_{i=1}^d{a_ib_i}$, while the
Frobenius inner product between two tensors $\bfA=(A_{ij})_{i,j\in \{1,\ldots,d\}},$ $\bfB=(B_{ij})_{i,j\in \{1,\ldots,d\}}\in \mathbb{R}^{d\times d }$ is denoted by
$\bfA: \bfB\coloneqq \sum_{i,j=1}^d{A_{ij}B_{ij}}$.  
Moreover, for a (Lebesgue) measurable set $M\subseteq \mathbb{R}^n$, $n\in \mathbb{N}$,  and (Lebesgue) measurable functions, vector or tensor field $\mathbf{u},\mathbf{w}\in (L^0(M))^{\ell}$, $\ell\in  \mathbb{N}$, we write $(\mathbf{u},\mathbf{w})_M \coloneqq \int_M \mathbf{u} \odot  \mathbf{w}\,\mathrm{d}x$, 
 whenever the right-hand side is well-defined, where $\odot\colon \mathbb{R}^{\ell}\times \mathbb{R}^{\ell}\to \mathbb{R}$ either denotes scalar multiplication, the Euclidean inner  product or the Frobenius inner product.
The integral mean  of an integrable function, vector or tensor field $\mathbf{u}\in (L^0(M))^{\ell}$, $\ell\in  \mathbb{N}$,  over a (Lebesgue) measurable set $M\subseteq \mathbb{R}^n$, $n\in \mathbb{N}$, is denoted by ${\mean{\mathbf{u}}_M \coloneqq \smash{\frac 1 {|M|}\int_M \mathbf{u} \, \mathrm{d}x}}$.

\subsection{N-functions and Orlicz spaces}

\hspace*{5mm}A convex function $\psi \colon \setR^{\geq 0} \to \setR^{\geq 0}$ is called \textit{N-function} if it holds that~${\psi(0)=0}$,~${\psi(t)>0}$~for~all~${t>0}$, $\lim_{t\rightarrow0} \psi(t)/t=0$, and
$\lim_{t\rightarrow\infty} \psi(t)/t=\infty$. 
A Carath\'eodory function $\psi \colon M \times \setR^{\geq 0} \to \setR^{\geq 0}$, where $M\subseteq \mathbb{R}^n$, $n\in \mathbb{N}$, is a (Lebesgue) measurable set, such that $\psi(x,\cdot)$ is an N-function for a.e.\ $x \in M$, is called \textit{generalized N-function}.
We~define the \textit{(convex) conjugate N-function} $\psi^*\colon M\times \setR^{\geq 0} \to \setR^{\geq 0}$ via
${\psi^*(t,x)\coloneqq \sup_{s \geq 0} (st - \psi(s,x))}$ for all $t \ge 0$ and a.e.\ $x\in M$, which satisfies
$(\partial_t(\psi^*))(x,t) =  (\partial_t\psi)^{-1}(x,t)$ for all $t\ge 0$ and a.e.\ $x\in M$. A (generalized) N-function
$\psi$ satisfies the \textit{$\Delta_2$-condition} (in short,
$\psi \in \Delta_2$), if there exists
$K> 2$ such that ${\psi(2\,t,x) \leq K\,\psi(t,x)}$ for all
$t \geq 0$ and a.e.\ $x\in M$. The smallest such constant is denoted by
$\Delta_2(\psi) > 0$.  If one assumes that
both~$\psi$~and~$\psi^*$ satisfy the $\Delta_2$-condition,~then~there~holds
\begin{align} 
  \label{eq:phi*phi'}
  \smash{\psi^*\circ\psi' \sim \psi\,.}
\end{align}
We need the following version of the \textit{$\epsilon$-Young inequality}: for every
${\epsilon> 0}$, there exists a constant $c_\epsilon>0 $,
depending only on $\Delta_2(\psi),\Delta_2( \psi ^*)<\infty$, such
that for every $s,t\geq 0$ and a.e.\ $x\in M$, it holds that
\begin{align} \label{ineq:young}
	s\,t&\leq c_\epsilon \,\psi^*(s,x)+ \epsilon \, \psi(t,x)\,.
\end{align}

\subsection{Basic properties of the extra stress tensor}

\hspace*{5mm}Throughout~the~entire~paper, we will always assume that the extra stress tensor 
$\SSS$
has \textit{$(p,\delta)$-structure}. A detailed
discussion and full proofs can be found, \textit{e.g.}, in
\cite{die-ett,dr-nafsa}. For a given tensor $\bfA\in \setR^{d\times d}$, we denote its symmetric part by
${\bfA^{\textup{sym}}\coloneqq[\bfA]^{\textup{sym}}\coloneqq  \frac{1}{2}(\bfA+\bfA^\top)\in
	\setR^{d\times d}_{\textup{sym}}\coloneqq \{\bfA\in \setR^{d\times
		d}\mid \bfA=\bfA^\top\}}$.

For $p \in (1,\infty)$~and~$\delta\ge 0$, we define the \textit{special N-function}
$\phi=\phi_{p,\delta}\colon\setR^{\ge 0}\to \setR^{\ge 0}$ via
\begin{align} 
	\label{eq:def_phi} 
	\varphi(t)\coloneqq  \int _0^t \varphi'(s)\, \mathrm ds,\quad\text{where}\quad
	\varphi'(t) \coloneqq  (\delta +t)^{p-2} t\,,\quad\textup{ for all }t\ge 0\,.
\end{align}
An important tool in our analysis play {\rm shifted N-functions}
$\{\psi_a\}_{\smash{a \ge 0}}$ (\textit{cf.}\ \cite{DK08,dr-nafsa}). For a given N-function $\psi\colon\mathbb{R}^{\ge 0}\to \mathbb{R}^{\ge
	0}$, we define the family  of \textit{shifted N-functions} ${\psi_a\colon\mathbb{R}^{\ge
		0}\to \mathbb{R}^{\ge 0}}$,~${a \ge 0}$,~by
\begin{align}
	\label{eq:phi_shifted}
	\psi_a(t)\coloneqq  \int _0^t \psi_a'(s)\, \mathrm ds\,,\quad\text{where }\quad
	\psi'_a(t)\coloneqq \psi'(a+t)\frac {t}{a+t}\,,\quad\textup{ for all }t\ge 0\,.
\end{align}

\begin{assumption}[Extra stress tensor]\label{assum:extra_stress} 
	We assume that the extra stress tensor $\SSS\colon\setR^{d\times d}\to \setR^{d\times d}_{\textup{sym}}$ belongs to $C^0(\setR^{d\times d}; \setR^{d\times d}_{\textup{sym}})\cap C^1(\setR^{d\times d}\setminus\{\mathbf{0}\}; \setR^{d\times d}_{\textup{sym}}) $ and satisfies $\SSS(\bfA)=\SSS(\bfA^{\textup{sym}})$ for all $\bfA\in \setR^{d\times d}$ and $\SSS(\mathbf{0})=\mathbf{0}$. Moreover, we assume~that~the~tensor $\SSS=(S_{ij})_{i,j=1,\ldots,d}$ has \textup{$(p,\delta)$-structure}, \textit{i.e.},
	for some $p \in (1, \infty)$, $ \delta\in [0,\infty)$, and the
	N-function $\varphi=\varphi_{p,\delta}$ (\textit{cf.}~\eqref{eq:def_phi}), there
	exist constants $C_0, C_1 >0$ such that
	\begin{align}
		\sum\limits_{i,j,k,l=1}^d \partial_{kl} S_{ij} (\bfA)
		B_{ij}B_{kl} &\ge C_0 \, \frac{\phi'(|\bfA^{\textup{sym}}|)}{|\bfA^{\textup{sym}}|}\,|\bfB^{\textup{sym}}|^2\,,\label{assum:extra_stress.1}
		\\
		\big |\partial_{kl} S_{ij}({\bfA})\big | &\le C_1 \, \frac{\phi'(|\bfA^{\textup{sym}}|)}{|\bfA^{\textup{sym}}|}\label{assum:extra_stress.2}
	\end{align}
	are satisfied for all $\bfA,\bfB \in \setR^{d\times d}$ with $\bfA^{\textup{sym}}\neq \mathbf{0}$ and all $i,j,k,l=1,\ldots,d$.~The~constants $C_0,C_1>0$ and $p\in (1,\infty)$ are called the \textup{characteristics} of $\SSS$. \enlargethispage{2mm}
\end{assumption}

\begin{remark}
	\begin{itemize}[{(ii)}]
		\item[(i)] Assume that $\SSS$ satisfies Assumption \ref{assum:extra_stress} for some 
		$\delta \in [0,\delta_0]$.~Then,~if~not~otherwise~stated, the
		constants in the estimates depend only on the characteristics~of~$\SSS$~and $\delta_0\ge 0$, but are independent of $\delta\ge 0$.
		
		\item[(ii)] Let $\phi$ be defined in \eqref{eq:def_phi} and 
		$\{\phi_a\}_{a\ge 0}$ be the corresponding family of the shifted \mbox{N-functions}. Then, the operators 
		$\SSS_a\colon\mathbb{R}^{d\times d}\to \smash{\mathbb{R}_{\textup{sym}}^{d\times
				d}}$, $a \ge 0$, defined, for every $a \ge 0$
		and~$\bfA \in \mathbb{R}^{d\times d}$, via 
		\begin{align}
			\label{eq:flux}
			\SSS_a(\bfA) \coloneqq 
			\frac{\phi_a'(\abs{\bfA^{\textup{sym}}})}{\abs{\bfA^{\textup{sym}}}}\,
			\bfA^{\textup{sym}}\,, 
		\end{align}
		have $(p, \delta +a)$-structure.  In this case, the characteristics of
		$\SSS_a$ depend~only~on~${p\in (1,\infty)}$ and are independent of
		$\delta \geq 0$ and $a\ge 0$.
	\end{itemize}
\end{remark}

Closely related to the extra stress tensor $\SSS\colon \setR^{d\times d}\to \setR^{d\times d}_{\textup{sym}}$ with
$(p,\delta)$-structure (\textit{cf.}\ Assumption \ref{assum:extra_stress}) is the non-linear mapping
$\bfF\colon\setR^{d\times d}\to \setR^{d\times d}_{\textup{sym}}$, defined, 
for every $\bfA\in \mathbb{R}^{d\times d}$, via 
\begin{align}
	\begin{aligned}
		\bfF(\bfA)&\coloneqq (\delta+\vert \bfA^{\textup{sym}}\vert)^{(p-2)/2}\bfA^{\textup{sym}}\,.
	\end{aligned}\label{eq:def_F}
\end{align}
The connections between
$\SSS,\bfF\colon \setR^{d \times d}
\to \setR^{d\times d}_{\textup{sym}}$ and
$\phi_a,(\phi_a)^*\colon \setR^{\ge
	0}\to \setR^{\ge
	0}$,~${a\ge  0}$, are best explained
by the following proposition (\textit{cf.}~\cite{die-ett,dr-nafsa,dkrt-ldg}).

\begin{proposition}
	\label{lem:hammer}
	Let $\SSS$ satisfy Assumption~\ref{assum:extra_stress}, let $\varphi$ be defined in \eqref{eq:def_phi}, and let $\bfF$ be defined in \eqref{eq:def_F}. Then, uniformly with respect to 
	$\bfA, \bfB \in \setR^{d \times d}$, we have that\vspace{-1mm}
	\begin{align}
          \big(\SSS(\bfA) - \SSS(\bfB)\big):(\bfA-\bfB )
          &\sim  \vert  \bfF(\bfA) - \bfF(\bfB)\vert ^2 \notag
          \\
          &\sim \phi_{\vert \bfA^{\textup{sym}}\vert }(\vert \bfA^{\textup{sym}}
            - \bfB^{\textup{sym}}\vert) \label{eq:hammera}
          \\
          &\sim(\varphi_{\vert\bfA^{\textup{sym}}\vert})^*(\abs{\SSS(\bfA ) - \SSS(\bfB )})\,,\notag
          \\
          \label{eq:hammere}
          \abs{\SSS(\bfA) - \SSS(\bfB)}
          &\sim  \smash{\phi'_{\abs{\bfA}}(\abs{\bfA - \bfB})}\,.
	\end{align}
	The constants 
	depend only on the characteristics of ${\SSS}$.
\end{proposition} 
\begin{remark}\label{rem:sa}
	{\rm
		For the operators $\SSS_a\hspace{-0.1em}\colon\hspace{-0.1em}\mathbb{R}^{d\times d}\hspace{-0.1em}\to\hspace{-0.1em}\smash{\mathbb{R}_{\textup{sym}}^{d\times
				d}}$, $a \ge  0$, defined~in~\eqref{eq:flux},~the~assertions of Proposition \ref{lem:hammer} hold with $\phi\colon\mathbb{R}^{\ge 0}\to \mathbb{R}^{\ge 0}$ replaced
		by $\phi_a\colon\mathbb{R}^{\ge 0}\to \mathbb{R}^{\ge 0}$, $a\ge 0$.}
\end{remark}

The following results can be found in~\cite{DK08,dr-nafsa}. 

\begin{lemma}[Change of Shift]\label{lem:shift-change}
	Let $\varphi$ be defined in \eqref{eq:def_phi} and let $\bfF$ be defined in \eqref{eq:def_F}. Then,
	for each $\varepsilon>0$, there exists $c_\varepsilon\geq 1$ (depending only
	on~$\varepsilon>0$ and the characteristics of $\phi$) such that for every $\bfA,\bfB\in\smash{\setR^{d \times d}_{\textup{sym}}}$ and $t\geq 0$, it holds that
	\begin{align}
		\smash{\phi_{\vert \bfB\vert}(t)}&\leq \smash{c_\varepsilon\, \phi_{\vert \bfA\vert}(t)
			+\varepsilon\, \vert \bfF(\bfB) - \bfF(\bfA)\vert ^2\,,}\label{lem:shift-change.1}
		\\
		\smash{\phi_{\vert \bfB\vert}(t)}&\leq \smash{c_\varepsilon\, \phi_{\vert \bfA\vert} (t)
			+\varepsilon\, \phi_{\vert \bfA\vert}(\vert \vert \bfB\vert - \vert \bfA\vert\vert )\,,}\label{lem:shift-change.2}
		\\
		\smash{(\phi_{\vert \bfB\vert})^*(t)}&\leq \smash{c_\varepsilon\, (\phi_{\vert \bfA\vert})^*(t)
			+\varepsilon\, \vert \bfF(\bfB) - \bfF(\bfA)\vert^2} \,,\label{lem:shift-change.3}
		\\
		\smash{(\phi_{\vert \bfB\vert})^*(t)}&\leq \smash{c_\varepsilon\, (\phi_{\vert \bfA\vert})^*(t)
			+\varepsilon\, \phi_{\vert \bfA\vert}(\vert \vert \bfB\vert - \vert \bfA\vert\vert )}\,.\label{lem:shift-change.4}
	\end{align}
\end{lemma}

\subsection{The $p$-Navier--Stokes system} 
\hspace*{5mm}Let us briefly recall some well-known facts about the $p$-Navier--Stokes equations
\eqref{eq:p-navier-stokes}. For $p\in (1,\infty)$, we define the function spaces\\[-4.5mm]
\begin{align*}
	\Vo\coloneqq (W^{1,p}_0(\Omega))^d\,,\qquad
	\Qo\coloneqq L_0^{p'}(\Omega)\coloneqq \big\{z\in
	L^{p'}(\Omega)\;|\;\mean{z}_{\Omega}=0\big\}\,.\\[-6mm] 
\end{align*}
With this notation, assuming that $p\ge \frac{3d}{d+2}$, the weak formulation of the $p$-Navier--Stokes equations \eqref{eq:p-navier-stokes} as a non-linear saddle point problem is the following:

\textit{Problem (Q).}\hypertarget{Q}{} For given $\bff \in (L^{p'}(\Omega))^d$, find $(\bfv,q)\in \Vo \times \Qo$ such that  for all $(\bfz,z)^\top\in \Vo\times Q $, it holds\linebreak\hspace*{4.5mm} that
\begin{align}
	(\SSS(\bfD\bfv),\bfD\bfz)_\Omega+([\nabla\bfv]\bfv,\bfz)_\Omega-(q,\divo\bfz)_\Omega&=(\bff ,\bfz)_\Omega\label{eq:q1}\,,\\
	(\divo\bfv,z)_\Omega&=0\label{eq:q2}\,.
\end{align}
Alternatively, we can reformulate Problem (\hyperlink{Q}{Q}) \textit{``hiding''} the kinematic pressure.

\textit{Problem (P).}\hypertarget{P}{} For given $\bff \in (L^{p'}(\Omega))^d$, find $\bfv\in \Vo(0)$ 
such that for all $\bfz\in \Vo(0) $, it holds that
\begin{align}
	(\SSS(\bfD\bfv),\bfD\bfz)_\Omega+([\nabla\bfv]\bfv,\bfz)_\Omega&=(\bff ,\bfz)_\Omega\,,\label{eq:p}
\end{align}
\hspace*{5mm}where $\Vo(0)\coloneqq \{\bfz\in \Vo\mid \divo \bfz=0\}$.

The names \textit{Problem (\hyperlink{Q}{Q})} and \textit{Problem (\hyperlink{P}{P})} are traditional in the literature (\textit{cf.}\  \cite{BF1991,bdr-phi-stokes}).~The~well-posed-ness of Problem (\hyperlink{Q}{Q}) and Problem (\hyperlink{P}{P}) is usually established in two steps:
first, using pseudo-monotone operator theory (\textit{cf.}\ \cite{lions-quel}), the well-posedness of Problem (\hyperlink{P}{P}) is shown; then, given the well-posedness of Problem (\hyperlink{P}{P}), the well-posedness of Problem (\hyperlink{Q}{Q})
follows using DeRham's lemma. There holds the following regularity property of the
pressure if the velocity satisfies a natural regularity assumption.
\begin{lemma}\label{lem:pres}
	Let $\SSS$ satisfy Assumption~\ref{assum:extra_stress} with
	$p\in (2,\infty)$ and $\delta >0$, and let $(\bfv,q)^\top \in \Vo(0)\times \Qo$
	be a weak solution of Problem (\hyperlink{Q}{Q}) with $\bfF(\bfD \bfv) \in
	(W^{1,2}(\Omega))^{d\times d} $. Then, the following statements apply:
	\begin{itemize}[{(ii)}]
		\item[(i)] If $\bff \in  (L^{p'}(\Omega))^d$, then $\nabla q \in  (L^{p'}(\Omega))^d$.
		\item[(ii)] If $\bff \in (L^2(\Omega))^d$, then $(\delta
			+\vert\bfD\bfv\vert)^{2-p}\vert\nabla q\vert ^2 \in  L^1(\Omega)$.
	\end{itemize}
\end{lemma}\vspace*{-5mm}
\begin{proof} See \cite[Lem.\ 2.6]{kr-pnse-ldg-2}.
\end{proof}

\newpage
\subsection{Discussion of Muckenhoupt regularity condition}\label{sec:reg-assumption}

\hspace*{5mm}In this subsection, we examine the \textit{Muckenhoupt regularity condition}
\begin{align}
	\mu_{\bfD\bfv} \coloneqq (\delta+\vert \overline{\bfD\bfv}\vert)^{p-2}\in A_2(\mathbb{R}^d)\,,\label{eq:reg-assumption}
\end{align}
on a solution $\bfv\in (W^{1,p}_0(\Omega))^d$ of Problem (\hyperlink{P}{P}) (or Problem (\hyperlink{Q}{Q}), respectively), where $\overline{\bfD\bfv}\in (L^p(\mathbb{R}^d))^{d\times d}$ is defined via\vspace*{-2mm}\enlargethispage{6mm}
\begin{align*}
	\overline{\bfD\bfv}\coloneqq\begin{cases}
		\bfD\bfv &\text{ a.e.\ in }\Omega\,,\\
		\bfzero  &\text{ a.e.\ in }\mathbb{R}^d\setminus\Omega\,.
	\end{cases}
\end{align*}
In this connection, recall that for given $p\in [1,\infty)$, a weight  $\sigma\colon \mathbb{R}^d\to (0,+\infty)$, \textit{i.e.}, $\sigma \in L^1_{\textup{loc}}(\mathbb{R}^d)$ and $0<\sigma(x)<+\infty$ for a.e.\ $x\in \mathbb{R}^d$, is said to satisfy the \textit{$A_p$-condition}, if
\begin{align*}
	[\sigma]_{A_p(\mathbb{R}^d)}\coloneqq 	\sup_{B\subseteq \mathbb{R}^d\,:\,B\text{ is a ball}}{\langle \sigma\rangle_B(\langle \sigma^{1-p'}\rangle_B)^{p-1}}<\infty\,.
\end{align*}
We denote by $A_p(\mathbb{R}^d)$ the \textit{class of all weights
	satisfying the $A_p$-condition} and set
      $A_\infty(\mathbb{R}^d) \coloneqq \bigcup
      _{p>1}A_p(\mathbb{R}^d)$. Moreover, use \textit{weighted Lebesgue
	spaces} $L^p(\Omega;\sigma)$ equipped with the norm
$\|\cdot\|_{p,\sigma,\Omega}\coloneqq  (\int_\Omega
\vert \cdot\vert^p\, \sigma \, \mathrm{d}x)^{1/p}$.
 
In two dimensions, the Muckenhoupt regularity condition \eqref{eq:reg-assumption} is satisfied under mild assumptions.

\begin{theorem}\label{thm:reg_two_dim}
		Let $\Omega\subseteq \mathbb{R}^2$  be a bounded domain with $C^2$-boundary, $p>2$, and $\bff \in (L^s(\Omega))^2$,~where~$s>2$. Then, there exist $q>2$, $\alpha>0$, and a solution $(\bfv,q)^\top\in \Vo(0)\times \Qo$ of Problem (\hyperlink{Q}{Q}) with the following properties:\vspace*{-2mm}
			\begin{itemize}[{(iii)}]
			\item[(i)] $(\bfv,q)^\top\in (W^{2,q}(\Omega))^d\times W^{1,q}(\Omega)$;
			\item[(ii)] $(\bfv,q)^\top\in (C^{1,\alpha}(\overline{\Omega}))^d\times C^{0,\alpha}(\overline{\Omega})$.
		\end{itemize}
\end{theorem}\vspace*{-5mm}

\begin{proof}
	See \cite[Thm. 6.1]{KMS2}.
\end{proof}

\begin{remark}\label{rem:muckenhoupt}
	\begin{itemize}[{(ii)}]
		\item[(i)] If $\vert \bfD\bfv\vert\in L^\infty(\Omega)$, $\delta>0$,  and $p>2$, then we have that $\delta^{p-2}\leq\mu_{\bfD\bfv}\leq (\delta+\|\bfD\bfv\|_{\infty,\Omega})^{p-2}$ a.e.\ in $\Omega$, 
		so that the Muckenhoupt regularity condition \eqref{eq:reg-assumption} is trivially satisfied. As a result, under the assumptions of Theorem \ref{thm:reg_two_dim}, the regularity assumption \eqref{eq:reg-assumption} is satisfied.
		\item[(ii)] We believe that it is possible to prove Theorem \ref{thm:reg_two_dim} without the $C^2$-boundary assumption for polygonal, convex domains.
	\end{itemize}
\end{remark}

The following result implies that, in three dimensions, one cannot hope for the regularity assumption to be satisfied, in general.

\begin{theorem}\label{thm:counter_example}
	Let $\Omega\subseteq \mathbb{R}^d$, $d\ge 2$,  be a bounded domain with $C^{1,1}$-boundary. Then, there exists a vector field $\bfv\colon \Omega\to \mathbb{R}^d$ with the following properties:\vspace*{-2mm}
	\begin{itemize}[{(iii)}]
		\item[(i)] $\bfv\in (W^{2,q}(\Omega))^d\cap (W^{1,s}_0(\Omega))^d$ for all $q\in (1,d)$ and $s\in (1,\infty)$;
		\item[(ii)] $\textup{tr}_{\partial\Omega}(\nabla \bfv) \not\equiv \bfzero$;
		\item[(iii)] $\divo \bfv=  0$ a.e.\ in $\Omega$;
		\item[(iv)] $\mu_{\bfD\bfv}\coloneqq(\delta+\vert \overline{\bfD\bfv}\vert)^{p-2}\notin A_\infty(\mathbb{R}^d)$ for all $p\in (1,\infty)\setminus\{2\}$ and $\delta\in [0,1]$.
	\end{itemize}
\end{theorem}\vspace*{-5mm}

\begin{proof}
 See \cite[Thm. 2.1]{kr-nekorn} and \cite[Thm. 2]{kr-nekorn-add}.
\end{proof}

\begin{remark}
		Let $\bfv \colon \Omega\to \mathbb{R}^3$ be the vector
                field from Theorem \ref{thm:counter_example} in the
                case $d = 3$. Setting ${\bff \coloneqq
                  \divo(\bfv\otimes\bfv)-\divo \bfS(\bfD\bfv)}$, where
                $\bfS(\bfD\bfv)\coloneqq (\delta+\vert \bfD\bfv\vert)^{p-2}\bfD\bfv$, we see that $\bfv$ and $q \equiv 0$ are a weak solution of the $p$-Navier--Stokes system \eqref{eq:p-navier-stokes} for any $p \in  (1, \infty)$ and $\delta\in [0,1]$. Due  to \cite[Rem.\ 2.3]{kr-nekorn} and \cite{kr-nekorn-add}, we have that $\bfF(\bfD\bfv)\in (W^{1,2}(\Omega))^{3\times 3}$ and $\bff \in (L^2(\Omega))^3$~for~all~$p \in (1, \infty)$.
\end{remark}

\subsection{DG spaces, jumps and averages}\label{sec:dg-space}

\subsubsection{Triangulations}\enlargethispage{4mm}

\hspace*{5mm}We always denote by $\{\pazocal{T}_h\}_{h>0}$ a family of regular (\textit{i.e.}, uniformly shape 
regular and conforming, \textit{cf.} \cite{BS08}) triangulations
of $\Omega\subseteq \setR^d$, $d\in \set{2,3}$, each
consisting of $d$-dimensional simplices $K$.  Here, the parameter
$h>0$, refers to the \textit{maximal mesh-size} of $\pazocal{T}_h$, for which we always assume that $h \le 1$. Moreover, we always assume that the chunkiness is bounded by some constant $\omega_0>0$, independent~of~$h$. By $\Gamma_h^{i}$, we denote the interior faces, and put
$\Gamma_h\coloneqq  \Gamma_h^{i}\cup \partial\Omega$. For an interior face
$\gamma= \partial K \cap \partial K' \in \Gamma^i_h$, we set $\omega_\gamma\coloneqq
K\cup K'$, and for a boundary face  $\gamma = K \cap \partial \Omega$, we set
$\omega_\gamma\coloneqq K$.  We
assume that each simplex $ K \in \pazocal{T}_h$ has at most one face from $\partial\Omega$.  
For (Lebesgue) measurable functions, vector or tensor fields $\mathbf{u},\mathbf{w}\in (L^0(\Gamma_h))^{\ell}$, $\ell\in  \mathbb{N}$, we write  
\begin{align*}
	\langle \mathbf{u},\mathbf{w}\rangle_{\Gamma_h} \coloneqq  \smash{\sum_{\gamma \in \Gamma_h} {\langle \mathbf{u}, \mathbf{w}\rangle_\gamma}}\,,\quad\text{ where }\quad\langle \mathbf{u}, \mathbf{w}\rangle_\gamma\coloneqq \int_\gamma \mathbf{u}\odot \mathbf{w} \,\textup{d}s\quad\text{ for all }\gamma\in \Gamma_h\,,
\end{align*}
whenever all the integrals are well-defined, where $\odot\colon \mathbb{R}^{\ell}\times \mathbb{R}^{\ell}\to \mathbb{R}$ either denotes scalar multiplication, the Euclidean inner  product or the Frobenius inner product.
Analogously, we define the products 
$\smash{\skp{\cdot}{\cdot}_{\partial\Omega}}$ and~$\smash{\skp{\cdot}{\cdot}_{\Gamma_h^{i}}}$. We extend the notation of
modulars to the sets $\smash{\Gamma_h^{i}}$, 
$\partial \Omega$, and $\smash{\Gamma_h}$, \textit{i.e.}, we
define the modulars ${\rho_{\psi,M}(\mathbf{u})\coloneqq  \smash{\int_M
		\psi(\abs{\mathbf{u}})\,\textup{d}s}}$ for all $\mathbf{u}\in \smash{(L^\psi(M))^{\ell}}$, $\ell\in  \mathbb{N}$, where $M= \smash{\Gamma_h^{i}}$,~${M=\partial \Omega}$,~or~${M=\smash{\Gamma_h}}$.

\subsubsection{Broken function spaces and projectors}
\hspace*{5mm}For every
$m \hspace{-0.1em}\in\hspace{-0.1em}
\setN_0$~and~${K\hspace{-0.1em}\in\hspace{-0.1em} \pazocal{T}_h}$, we
denote by $\mathbb{P}_m(K)$ the space of polynomials of
degree at most $m$ on $K$. Then, for given $k \in \setN_0$ and $p\in
(1,\infty)$,  
we define the spaces
\begin{align}
	\begin{split}
		Q_h^k&\coloneqq \big\{ z_h\in L^1(\Omega)\;\big |\; z_h|_K\in \mathbb{P}_k(K)\text{ for all }K\in \pazocal{T}_h\big\}\,,\\
		V_h^k&\coloneqq \big\{\bfz_h\in (L^1(\Omega))^d\;\big |\; \bfz_h|_K\in (\mathbb{P}_k(K))^d\text{ for all }K\in \pazocal{T}_h\big\}\,,\\
		X_h^k&\coloneqq \big\{\bfX_h\in (L^1(\Omega))^{d\times d}\;\big |\; \bfX_h|_K\in (\mathbb{P}_k(K))^{d\times d}\text{ for all }K\in \pazocal{T}_h\big\}\,,\\
		W^{1,p}(\mathcal T_h)&\coloneqq \big\{ w_h\in
                L^1(\Omega)\;\big |\; w_h|_K\in W^{1,p}(K)\text{ for all }K\in \pazocal{T}_h\big\}\,.
	\end{split}\label{eq:2.19}
\end{align}
In addition, for given $k \in \setN_0$, we set $\Qhkc
\coloneqq Q_h^k\cap C^0(\overline{\Omega})$.
Note that $ {W^{1,p}(\Omega)\subseteq
  \WDG}$ and that ${Q_h^k\subseteq \WDG}$.
We denote by $\PiDG \colon (L^1(\Omega))^d\to V_h^k$, the \textit{(local)
$L^2$-projection} into $V_h^k$, which for every $\bfv \in
(L^1(\Omega))^d$ and $\bfz_h
\in V_h^k$ is defined via  $(\PiDG \bfv,\bfz_h)_\Omega=(\bfv,\bfz_h)_\Omega$. 
Similarly, we define the (local) \mbox{$L^2$-projection}~into~$\Xhk$, \textit{i.e.}, $\PiDG\colon (L^1(\Omega))^{d\times d} \to \Xhk$.

For every  $\bfw_h\in (\WDG)^d$, we denote by $\nabla_h \bfw_h\in (L^p(\Omega))^{d\times d}$,
the element-wise~gradient, defined via
$(\nabla_h \bfw_h)|_K\hspace*{-0.1em}\coloneqq\hspace*{-0.1em} \nabla(\bfw_h|_K)$
for~all~${K\hspace*{-0.1em}\in\hspace*{-0.1em}\pazocal{T}_h}$.  Then, given a multiplication operator~${\odot\colon \hspace*{-0.1em}\mathbb{R}^d\hspace*{-0.1em}\times \hspace*{-0.1em}\mathbb{R}^d\hspace*{-0.1em}\to\hspace*{-0.1em} \mathbb{R}^{\ell}}$,~${m,l\hspace*{-0.1em}\in \hspace*{-0.1em}\setN}$, for
every $\bfw_h\in (\WDG)^d$ and interior faces $\gamma\in \Gamma_h^{i}$ shared by
adjacent elements $K^-_\gamma, K^+_\gamma\in \pazocal{T}_h$, we define
via
\begin{align}
	\{\bfw_h\}_\gamma&\coloneqq \tfrac{1}{2}(\textup{tr}_\gamma^{K^+}(\bfw_h)+
	\textup{tr}_\gamma^{K^-}(\bfw_h))\in
	(L^p(\gamma))^{\ell}\,, \label{2.20}\\
	\llbracket\bfw_h\odot \bfn\rrbracket_\gamma
	&\coloneqq \textup{tr}_\gamma^{K^+}(\bfw_h)\odot \bfn^+_\gamma+
	\textup{tr}_\gamma^{K^-}(\bfw_h)\odot \bfn_\gamma^- 
	\in (L^p(\gamma))^{\ell}\,,\label{eq:2.21}
\end{align}
the \textit{average} and \textit{normal jump}, respectively, of $\bfw_h$ on $\gamma$.
Moreover,  for boundary faces $\gamma\in \partial\Omega$, we define \textit{boundary averages} and 
\textit{boundary jumps}, respectively, via
\begin{align}
	\{\bfw_h\}_\gamma&\coloneqq \textup{tr}^\Omega_\gamma(\bfw_h) \in (L^p(\gamma))^{\ell}\,,\label{eq:2.23a} \\
	\llbracket \bfw_h\odot\bfn\rrbracket_\gamma&\coloneqq 
	\textup{tr}^\Omega_\gamma(\bfw_h)\odot\bfn \in (L^p(\gamma))^{\ell}\,,\label{eq:2.23} 
\end{align}
where $\bfn\colon\partial\Omega\to \mathbb{S}^{d-1}$ denotes the unit normal vector field to $\Omega$ pointing outward. 
If there is no
danger of confusion, then we will omit the index $\gamma\in
\Gamma_h$,~in~particular,  if we interpret jumps and averages as
global functions defined on whole $\Gamma_h$. 

\subsubsection{DG gradient, symmetric gradient, divergence and jump operators}

\hspace*{5mm}For every $k\in \mathbb{N}_0$,  the  \textit{lifting operator} $\setRhk\colon(\WDG)^d \to X_h^k$, for every 
 $\bfw_h\in \smash{(\WDG)^d}$ and  $\bfX_h\in X_h^k$, (via~Riesz~representation) is defined via
\begin{align}
	(\setRhk\bfw_h,\bfX_h)_\Omega=\langle
		\llbracket\bfw_h\otimes\bfn\rrbracket,\{\bfX_h\}\rangle_{\Gamma_h}\,.\label{eq:2.25.1}
\end{align}
For $k\in \mathbb{N}_0$, the \textit{DG gradient operator} $  \Ghk\colon (\WDG)^d\to (L^p(\Omega))^{d\times d}$, the \textit{DG symmetric gradient operator} $  \Dhk\colon \hspace*{-0.15em}(\WDG)^d\hspace*{-0.15em}\to\hspace*{-0.15em} (L^p(\Omega))^{d\times d}$, and the \textit{DG divergence~operator}~$ {\Divhk\colon \hspace*{-0.15em}(\WDG)^d\hspace*{-0.15em}\to\hspace*{-0.15em} L^p(\Omega)}$,  for every $\bfw_h\in \smash{(\WDG)^d}$, are defined via
\begin{align*}
	\begin{aligned}
		\Ghk\bfw_h&\coloneqq \nabla_h \bfw_h -\setRhk \bfw_h&&\quad\text{ in }(L^p(\Omega))^{d\times d}\,,\\
		\Dhk\bfw_h&\coloneqq [\Ghk\bfw_h]^{\textup{sym}}&&\quad\text{ in }(L^p(\Omega))^{d\times d}\,,\\
		\Divhk\bfw_h&\coloneqq \textrm{tr}(\Ghk\bfw_h)&&\quad\text{ in }L^p(\Omega)\,.
	\end{aligned}
\end{align*}
In addition, for every $\bfw_h\in \smash{(\WDG)^d}$, we introduce the \textit{DG norm} and the \textit{symmetric DG norm} as
\begin{align}
	\|\bfw_h\|_{\nabla,p,h}&\coloneqq \|\nabla_h\bfw_h\|_{p,\Omega}+h^{1/p}\|h^{-1}\jump{\bfw_h\otimes \bfn}\|_{p,\Gamma_h}\,,\\
	\|\bfw_h\|_{\bfD,p,h}&\coloneqq \|\bfD_h\bfw_h\|_{p,\Omega}
	+h^{1/p}\|  h^{-1} \llbracket\bfw_h\otimes\bfn\rrbracket\|_{p,\Gamma_h}\,.
\end{align}
Owing to \cite[(A.26)--(A.28)]{dkrt-ldg} and \cite[Prop. 2.5]{kr-pnse-ldg-1}, there exists a constant $c>0$, independent of $h>0$, such that for every $\bfw_h\in \smash{(\WDG)^d}$, it holds that  
\begin{align}
   c^{-1}\,\|\bfw_h\|_{\nabla,p,h}\leq \|\Ghk\bfw_h\|_{p,\Omega}+h^{1/p}\|h^{-1}\jump{\bfw_h\otimes \bfn}\|_{p,\Gamma_h}\leq c\,\|\bfw_h\|_{\nabla,p,h}\label{eq:eqiv0}\,,\\
   c^{-1}\,\|\bfw_h\|_{\bfD,p,h}\leq \|\Dhk\bfw_h\|_{p,\Omega}+h^{1/p}\|h^{-1}\jump{\bfw_h\otimes \bfn}\|_{p,\Gamma_h}\leq c\,\|\bfw_h\|_{\bfD,p,h}\label{eq:eqiv0.1}\,.
\end{align}
Owing to \cite[Prop. 2.4]{kr-pnse-ldg-1}, there exists a constant $c>0$, independent of $h>0$,  such that for every $\bfz_h\in \Vhk$, it holds that
\begin{align}\label{discrete_korn}
		\|\bfz_h\|_{\nabla,p,h}\leq  c\,\|\bfz_h\|_{\bfD,p,h}\,.
\end{align}
For a generalized N-function $\psi\colon \Omega\times \setR^{\ge 0}\to \setR^{\ge 0}$,  the pseudo-modular\footnote{The
		definition of an pseudo-modular can be found in \cite{Mu}. We
	extend the notion of DG Sobolev spaces to  DG Sobolev-Orlicz
	spaces $W^{1,\psi}(\mathcal T_h)\coloneqq \big\{w_h\in L^1(\Omega)\mid w_h|_K\in W^{1,\psi}(K)\text{ for all }K\in \pazocal{T}_h\big\}$.} ${m_{\psi,h}\colon \hspace*{-0.15em}(W^{1,\psi}(\mathcal T_h))^d\hspace*{-0.15em}\to\hspace*{-0.15em} \mathbb{R}^{\ge 0}}$  and the modular $M_{\psi,h}\colon (W^{1,\psi}(\mathcal T_h))^d\to \mathbb{R}^{\ge 0}$, for every $\bfw_h\in (W^{1,\psi}(\mathcal T_h))^d$, are defined via
\begin{align} 
	m_{\psi,h}(\bfw_h)&\coloneqq  h\,\rho_{\psi,\smash{\Gamma_h}}(h^{-1}\jump{\bfw_h\otimes \bfn})\,,\label{def:mh.1}\\
	M_{\psi,h}(\bfw_h)&\coloneqq  \rho_{\psi,\Omega}(\nabla_h \bfw_h)+m_{\psi,h}(\bfw_h)\,. \label{def:mh.2}
\end{align}
For $\psi = \phi_{p,0}$, for every ${\bfw_h\in (W^{1,\psi}(\mathcal T_h))^d}$, it holds that
\begin{align*}
	m_{\psi,h}(\bfw_h)&=h\,\|h^{-1}\jump{\bfw_h\otimes
		\bfn}\|_{p,\Gamma_h}^p\,,\\
	M_{\psi,h}(\bfw_h)&= \|\bfw_h\|_{\nabla,p,h}^p\,.
\end{align*}

\section{Local discontinuous Galerkin (LDG) approximation}\label{sec:ldg}

\hspace*{5mm}In this section, we recall the LDG formulation proposed in \cite{kr-pnse-ldg-1,kr-pnse-ldg-2,kr-pnse-ldg-3} and the a priori error estimates derived therein. In addition, we derive a new quasi-optimal a priori error estimate for the kinematic pressure.

\subsection{LDG formulations}\label{sec:fluxed}

\hspace*{5mm}Appealing to \cite{kr-pnse-ldg-1,kr-pnse-ldg-2,kr-pnse-ldg-3}, an LDG formulation of Problem (\hyperlink{Q}{Q}) reads:

\textit{Problem (Q$_h$).}\hypertarget{Qhldg}{} For given $\bff \in (L^{p'}(\Omega))^d$ and $\alpha>0$, find $(\bfv_h,q_h)^\top\in \Vhk\times \Qhkco$ such that for every $ (\bfz_h,z_h)^\top \in \Vhk\times\Qhkc$, it holds that
\begin{align}\label{eq:primal1}
	\begin{aligned}
	(\SSS(\Dhk \bfv_h)-\tfrac{1}{2}\bfv_h\otimes \bfv_h-q_h\mathbf{I}_d,\Dhk
		\bfz_h)_\Omega
	&=     (\bff -\tfrac{1}{2}[\Ghk \bfv_h]\bfv_h,\bfz_h)_\Omega
	\\
	&\quad - \alpha \,\langle\SSS_{\smash{\sssl}}(h^{-1} \jump{\bfv_h\otimes
		\bfn}), \jump{\bfz_h \otimes \bfn}\rangle_{\Gamma_h}\,,\\
	(\Divhk \bfv_h,z_h)_\Omega&=0\,,
\end{aligned}
\end{align}   
where $\Qhkco\coloneqq \Qhkc\cap L^1_0(\Omega)$. The LDG formulation of Problem (\hyperlink{P}{P}) reads:

\textit{Problem (P$_h$).}\hypertarget{Phldg}{} For given $\bff \in (L^{p'}(\Omega))^d$ and $\alpha>0$, find $\bfv_h\in V_h^k(0)$ such that for every $\bfz_h\in V_h^k(0)$, it holds that
\begin{align}\label{eq:primal2}
	\begin{aligned}
		(\SSS(\Dhk \bfv_h)-\tfrac{1}{2}\bfv_h\otimes \bfv_h,\Dhk
		\bfz_h)_\Omega&=     (\bff -\tfrac{1}{2}[\Ghk \bfv_h]\bfv_h,\bfz_h)_\Omega\\
	&\quad - \alpha \,\langle\SSS_{\smash{\sssl}}(h^{-1} \jump{\bfv_h\otimes
		\bfn}), \jump{\bfz_h \otimes \bfn}\rangle_{\Gamma_h}\,,
	\end{aligned}
\end{align} 
where $V_h^k(0)\coloneqq \{\bfz_h \in \Vhk \mid (\Divhk \bfz_h,z_h)_\Omega=0\text{ for all }z_h\in \Qhkc\}$. 

Well-posedness (\textit{i.e.}, solvability), stability (\textit{i.e.}, a priori
estimates), and (weak) convergence of solutions of Problem (\hyperlink{Qhldg}{Q$_h$}) and Problem
(\hyperlink{Phldg}{P$_h$}) to solutions of Problem (\hyperlink{Q}{Q}) and Problem~(\hyperlink{P}{P}), respectively, are proved in  \cite{kr-pnse-ldg-1}.

\subsection{Quasi-optimal a priori error estimate for the kinematic pressure of shear-thickening fluids}\label{sec:rates}

\hspace*{5mm}In order to derive a priori error for the kinematic pressure, it is
first necessary to derive a priori error estimates for the velocity
vector field. In this respect, we  summarize the relevant results of the recent
contributions  \cite{kr-pnse-ldg-2} and \cite{kr-pnse-ldg-3} in the following theorem: 
\begin{theorem}
	\label{thm:error_LDG}
	Let 
	$p\hspace*{-0.1em}\in\hspace*{-0.1em}(2,\infty)$, $\delta\hspace*{-0.1em}>\hspace*{-0.1em} 0$, $k\hspace*{-0.1em}\in\hspace*{-0.1em} \mathbb{N}$, $\alpha\hspace*{-0.1em}>\hspace*{-0.1em}0$, and 
	$\bff \hspace*{-0.1em}\in\hspace*{-0.1em} (L^{p'}(\Omega))^d$. Moreover,~let~${\bfF(\bfD \bfv) \hspace*{-0.1em}\in\hspace*{-0.1em} (W^{1,2}(\Omega))^{d\times d}}$.   Then,
	there exists a constant $c_0 >0$, depending only on the characteristics of
	$\SSS$, $\delta^{-1}$, $k$, 
	$\alpha^{-1}$, and $\omega_0$, such that if $\|\nabla\bfv\|_{2,\Omega}\le c_0$, then, it holds that
	\begin{align*}
		\|\bfF(\Dhk \bfv_h)-\bfF(\bfD
			\bfv)\|_{2,\Omega}^2 +
		m_{\phi_{\smash{\sssl}},h } (\bfv_h-\bfv)&\leq c_1\, h^2\, \|
			\nabla \bfF(\bfD \bfv) \|_{2,\Omega}^2+c_1\,\rho_{(\phi_{\abs{\bfD \bfv}})^*,\Omega}(h\nabla q)\,,\\
		\| q_h-q\|_{p',\Omega}^2&\leq c_2\, (h^2 +\rho_{(\phi_{\abs{\bfD \bfv}})^*,\Omega}(h\nabla q))\,,
	\end{align*}
	where $c_1>0$ depends only on the characteristics of
	$\SSS$, $\delta^{-1}$, $k$, 
	$\alpha^{-1}$, $\omega_0$, $\|\bfv\|_{\infty,\Omega}$,  and $c_0$, and $c_2>0$, in addition, on $\|\bfF(\bfD\bfv)\|_{1,2,\Omega}$, $\|\nabla q\|_{p',\Omega}$,
	and 
	$\vert \Omega\vert $.
\end{theorem}

An immediate consequence of Theorem \ref{thm:error_LDG} is the following result.
 
\begin{corollary}\label{cor:error_LDG}
	Let the assumptions of Theorem \ref{thm:error_LDG} be satisfied. Then,~it~holds that
	\begin{align*}
		\norm{\bfF(\Dhk \bfv_h) - \bfF(\bfD
			\bfv)}_2^2 +
		\,m_{\phi_{\smash{\sssl}},h } (\bfv_h-\bfv)&\leq  c_1\, h^2\, \|\nabla\bfF(\bfD \bfv)\|_{2,\Omega}^2+c_1\,h^{p'}\!\rho_{\phi^*,\Omega}(\nabla q)\,,\\
			\| q_h-q\|_{p',\Omega}^2&\leq c_2\, h^{p'}\,,
	\end{align*}
	where $c_1>0$ depends only on the characteristics of
	$\SSS$, $\delta^{-1}$, $k$, 
	$\alpha^{-1}$, $\omega_0$, $\|\bfv\|_{\infty,\Omega}$,  and $c_0$, and $c_2>0$, in addition, on $\|\bfF(\bfD\bfv)\|_{1,2,\Omega}$, $\|\nabla q\|_{p',\Omega}$,
	and 
	$\vert \Omega\vert $. 
	If, in addition, $\bff \in (L^{2}(\Omega))^d$,~then, it holds that
	\begin{align*}
		\|\bfF(\Dhk \bfv_h) - \bfF(\bfD
				\bfv)\|_{2,\Omega}^2 +m_{\phi_{\smash{\sssl}},h } (\bfv_h-\bfv)
	&\leq  c_1\, h^2 \smash{\big ( \|\nabla\bfF(\bfD \bfv)
			\|_{2,\Omega}^2+\|\nabla q\|_{2,\mu_{\bfD\bfv}^{-1},\Omega}^2
			\big )
		}\,,\\
		\| q_h-q\|_{p',\Omega}&\leq c_2\, h\,,
	\end{align*}
	where $\mu_{\bfD\bfv} \coloneqq (\delta+\vert\overline{\bfD\bfv}\vert )^{p-2}\colon \mathbb{R}^d\to (0,+\infty)$ and 
 	$c_1>0$ depends only on the characteristics of
	$\SSS$, $\delta^{-1}$, $k$, 
	$\alpha^{-1}$, $\omega_0$, $\|\bfv\|_{\infty,\Omega}$,  and $c_0$, and $c_2>0$, in addition, on $\|\bfF(\bfD\bfv)\|_{1,2,\Omega}$, $\|\nabla q\|_{p',\Omega}$, and $\|\nabla q\|_{2,\mu_{\bfD\bfv}^{-1},\Omega}$.
\end{corollary}

\begin{remark}
	\begin{itemize}[{(ii)}]
		\item[(i)] Note that in Theorem \ref{thm:error_LDG}, the assumption $\bff \in (L^{p'}(\Omega))^d$ is equivalent~to~${\nabla q\in (L^{p'}(\Omega))^d}$ and in Corollary \ref{cor:error_LDG},  the assumption $\bff \in (L^2(\Omega))^d$ can be replaced by $ \nabla q\in (L^2(\Omega;\mu_{\bfD\bfv}^{-1}))^d$.
		
		\item[(ii)] To show that $\bfF(\bfD \bfv) \in
		(W^{1,2}(\Omega) )^{d\times d}$ holds under reasonable assumptions is still an open problem
		(\textit{cf.}~\cite{hugo-petr-rose} for partial results). However, this
		regularity is natural for elliptic problems of $p$-Laplace type
		(\textit{cf.}~\cite{gia-mod-86,giu1}) and proved in the two-dimensional case
		(\textit{cf.}~\cite{KMS2}) and in any dimension in the space periodic setting,
		since it follows from interior regularity (\textit{cf.}~\cite{hugo-petr-rose}).
	\end{itemize} 
\end{remark}

\begin{remark}
	\begin{itemize}[{(ii)}]
		\item[(i)]   In \cite{kr-pnse-ldg-2}, numerical experiments confirmed the quasi-optimality of the a priori error estimates for the velocity vector field  in Corollary \ref{cor:error_LDG}.
	
	\item[(ii)]   In \cite{kr-pnse-ldg-3}, numerical experiments indicated the sub-optimality of the a priori error estimates for the kinematic pressure in Corollary \ref{cor:error_LDG} in two dimensions.
		\end{itemize} 
\end{remark}

Let us continue with the main result of this paper. It proves the
conjecture from \cite{kr-pnse-ldg-3} under a mild additional
assumption on the velocity vector field, \textit{i.e.}, the Muckenhoupt regularity condition \eqref{eq:reg-assumption}.\enlargethispage{7mm}

\begin{theorem}
	\label{thm:error_pressure}
	Let  
	$p\hspace*{-0.1em}\in\hspace*{-0.1em}(2,\infty)$, $\delta\hspace*{-0.1em}>\hspace*{-0.1em} 0$, $k\hspace*{-0.1em}\in\hspace*{-0.1em} \mathbb{N}$, $\alpha\hspace*{-0.1em}>\hspace*{-0.1em}0$, and 
	$\bff\hspace*{-0.1em} \in \hspace*{-0.1em}(L^{p'}(\Omega))^d$. Moreover, let 
	${\bfF(\bfD \bfv) \hspace*{-0.1em}\in\hspace*{-0.1em} (W^{1,2}(\Omega))^{d\times d}}$  and $\mu_{\bfD\bfv}\coloneqq (\delta+\vert \overline{\bfD\bfv}\vert)^{p-2}\in A_2(\mathbb{R}^d)$.
 	 Then,
	there exists a constant $c_0 >0$, depending only on the characteristics of
	$\SSS$, $\delta^{-1}$, 
	$k$, 
	$\alpha^{-1}$, and $\omega_0$, such that if $\|\nabla\bfv\|_{2,\Omega}\le c_0$, then, it holds that
	\begin{align*}
		\rho_{(\varphi_{\smash{\vert\bfD\bfv\vert}})^*,\Omega}(q_h-q)\leq c\,h^2 \,\|\nabla\bfF(\bfD\bfv)\|_{2,\Omega}^2+c\,\rho_{(\varphi_{\smash{\vert
					\bfD\bfv\vert}})^*,\Omega}(h\,\nabla
		q)
		+c\,\inf_{z_h\in \Qhkco}{\rho_{(\varphi_{\vert \bfD\bfv\vert})^*,\Omega}(q-z_h)}\,,
	\end{align*}
	where $c>0$ depends only on the characteristics of
	$\SSS$, $\delta^{-1}$,  $\omega_0$,
	$\alpha^{-1}$, $k$, $\|\bfv\|_{\infty,\Omega}$,
and $c_0$.
\end{theorem}

An immediate consequence of Theorem \ref{thm:error_pressure} is the following result.

\begin{corollary}\label{cor:error_pressure}
	Let the assumptions of Theorem \ref{thm:error_pressure} be satisfied. Then,~it~holds that
	\begin{align*}
		\rho_{(\varphi_{\smash{\vert\bfD\bfv\vert}})^*,\Omega}(q_h-q)\leq  c\, h^2\, \|\nabla\bfF(\bfD \bfv)
		\|_{2,\Omega}^2+c\,h^{p'}\!\rho_{\phi^*,\Omega}(\nabla q)\,,
	\end{align*}
	where $c>0$ is a constant depending only   on the characteristics of
	$\SSS$, $\delta^{-1}$,  $\omega_0$,
	$\alpha^{-1}$, $k$, $\|\bfv\|_{\infty,\Omega}$,
	and $c_0$. If, in addition, $\bff \in (L^{2}(\Omega))^d$,~then, it holds that
	\begin{align*}
		\rho_{(\varphi_{\smash{\vert\bfD\bfv\vert}})^*,\Omega}(q_h-q)
		\leq  c\, h^2 \,\big ( \|\nabla\bfF(\bfD \bfv)
			\|_{2,\Omega}^2+\|\nabla q\|_{2,\mu_{\bfD\bfv}^{-1},\Omega}^2
			\big )\,,
	\end{align*}
	where $c>0$ is a constant depending only   on the characteristics of
	$\SSS$, $\delta^{-1}$,  $\omega_0$,
	$\alpha^{-1}$, $k$, $\|\bfv\|_{\infty,\Omega}$,
	and $c_0$.
\end{corollary}

The key ingredient in the proof of Theorem \ref{thm:error_pressure} is the following discrete convex conjugation inequality.

\begin{lemma}[Discrete convex conjugation inequality]\label{lem:discrete_convex_conjugation_ineq}
	Let $p\in [2,\infty)$ and $\delta\ge  0$. Moreover, let $\bfA\in (L^p(\Omega))^{d\times d}\cap (W^{1,1}(\Omega))^{d\times d}$ with 
	$\mu_\bfA \coloneqq (\delta+\vert\overline{\bfA}\vert)^{p-2}\in A_2(\mathbb{R}^d)$ and $\bfF(\bfA)\in (W^{1,2}(\Omega))^{d\times d}$. Then, there exists a constant $c>0$, depending on $k$, $\omega_0$, $p$, $\delta$, $\Omega$, and $[\mu_\bfA]_{A_2(\mathbb{R}^d)}$,~such~that~for~every~${z_h \in \Qhkco}$,~it~holds that 
	\begin{align*}
		\rho_{(\varphi_{\smash{\vert\bfA\vert}})^*,\Omega}(z_h )\leq \sup_{\bfz_h\in \Vhk}{[(z_h ,\Divhk\,\bfz_h)_\Omega-\smash{\tfrac{1}{c}}\,M_{\varphi_{\smash{\vert\bfA\vert}},h}(\bfz_h)]}+c\,h^2\,\|\nabla \bfF(\bfA)\|_{2,\Omega}^2\,.\\[-8mm]
	\end{align*}
\end{lemma}

The proof of Lemma \ref{lem:discrete_convex_conjugation_ineq} is based on two key ingredients.
The first key ingredient is the following  continuous counterpart.

\begin{lemma}[Convex conjugation inequality]\label{lem:convex_conjugation_ineq}
	Let   $p\in [2,\infty)$ and $\delta\ge 0$. Moreover, let $ \bfA\in (L^p(\Omega))^{d\times d}$ with 
	$\mu_\bfA \coloneqq (\delta+\vert\overline{\bfA}\vert)^{p-2}\in A_2(\mathbb{R}^d)$. Then, there exists a constant $c>0$, depending on $p$, $\Omega$, and $[\mu_\bfA]_{A_2(\mathbb{R}^d)}$, such that for every $z\in \Qo$, it holds that 
	\begin{align*}
		\rho_{(\varphi_{\smash{\vert\bfA\vert}})^*,\Omega}(z)\leq \sup_{\bfz\in \Vo}{[(z,\textup{div}\,\bfz)_\Omega-\smash{\tfrac{1}{c}}\,\rho_{\varphi_{\smash{\vert\bfA\vert}},\Omega}(\nabla \bfz)]}\,.\\[-8mm]
	\end{align*}
\end{lemma}

The second key ingredient is the following stability result for the (local) $L^2$-projection operators $\{\Pi_h^k\}_{h>0}$ in terms of the shifted modular $\rho_{\varphi_{\smash{\vert \bfA\vert}},\Omega}$.\enlargethispage{6mm}

	\begin{lemma}[Shifted modular estimate for $\Pi_h^k$]\label{lem:shifted_modular_Pi_div}
Let  $p\in (1,\infty)$ and $\delta\ge  0$.  
Then, there exists a constant $c>0$, depending on $k$, $\omega_0$,
$p$, and $\Omega$,  
such that for every $\bfz\in \Vo$ and every $\bfA \in
(L^{p}(\Omega))^{d\times d}\cap (W^{1,1}(\Omega))^{d\times d}$ with $\bfF(\bfA) \in (W^{1,2}(\Omega) )^{d\times d}$, it holds that
\begin{align*}
	M_{\varphi_{\smash{\vert\bfA\vert}},h}( \PiDG \bfz)\leq c\, \rho_{\varphi_{\smash{\vert\bfA\vert}},\Omega}(\nabla\bfz)+c\,h^2\,\|\nabla \bfF(\bfA)\|_{2,\Omega}^2\,.\\[-8mm]
\end{align*}
\end{lemma}

\begin{proof}
The assertion follows using \cite[Prop.~4.10 (4.22)]{kr-pnse-ldg-2}, that $m_{\varphi_{\smash{\vert\bfA\vert}},h}( \PiDG \bfz)=
m_{\varphi_{\smash{\vert\bfA\vert}},h}( \PiDG \bfz-\bfz)$ since ${\bfz \in \Vo}$, 
and that  
for every $K\in \mathcal{T}_h$, it holds that
\begin{align*}
	\rho_{ \varphi_{\smash{\vert\bfA\vert}},K}(\nabla \PiDG \bfz)&\leq c\,\rho_{ \varphi_{\smash{\vert\mean{\bfA}_K\vert}},K}(\nabla \PiDG \bfz) 
	+ c\,\|\bfF(\bfA)-\bfF(\mean{\bfA}_K)\|_{2,K}^2
  \\
  &\le c\,\rho_{ \varphi_{\smash{\vert\mean{\bfA}_K\vert}},K}(\nabla  \bfz)  + c\,\|\bfF(\bfA)-\bfF(\mean{\bfA}_K)\|_{2,K}^2
  \\
  &\le 	c\,\rho_{ \varphi_{\smash{\vert\bfA\vert}},K}(\nabla \bfz) + c\,\|\bfF(\bfA)-\bfF(\mean{\bfA}_K)\|_{2,K}^2
  \\
  &\le 	c\,\rho_{ \varphi_{\smash{\vert\bfA\vert}},K}(\nabla \bfz)  + c\,h^2\,\|
   \nabla \bfF(\bfA)\|_{2,K}^2\,,
\end{align*}
where we used twice Lemma \ref{lem:shift-change},
\cite[Lem.~A.4]{bdr-phi-stokes}, and Poincar\'e's inequality.
\end{proof}
 
To prove
Lemma \ref{lem:convex_conjugation_ineq}, we  need the
following stability result  for the Bogovski\u{\i} operator in terms
of the shifted modular~$\rho_{\varphi_{\smash{\vert \bfA\vert}},\Omega}$. 

\begin{lemma}[Shifted modular estimate for Bogovski\u{\i}'s operator]\label{lem:key}
	Let $p\in [2,\infty)$ and $\delta\ge 0$. Moreover, let $\bfA\in (L^p(\Omega))^{d\times d}$ with  $\mu_\bfA \coloneqq (\delta+\vert\overline{\bfA}\vert)^{p-2}\in A_2(\mathbb{R}^d)$. 
	Then, the Bogovski\u{\i}'s~operator~$\mathcal{B}_{\textup{Bog}}\colon C^\infty_{0,0}(\Omega)\to (C^\infty_0(\Omega))^d$  uniquely extends to a continuous operator from $L^{\varphi_{\smash{\vert\bfA\vert}}}_0(\Omega)$ to $(W^{1,\smash{\varphi_{\smash{\vert\bfA\vert}}}}_0(\Omega))^d$. In particular,  there exists a constant $c>0$, depending on $p$, $\Omega$, and $[\mu_\bfA]_{A_2(\mathbb{R}^d)}$, such that for every $z\in L^{\varphi_{\smash{\vert\bfA\vert}}}_0(\Omega)$, it holds that
	\begin{align*}
		\rho_{\varphi_{\smash{\vert\bfA\vert}},\Omega}(\nabla\mathcal{B}_{\textup{Bog}}z)\leq c\,\rho_{\varphi_{\smash{\vert\bfA\vert}},\Omega}(z)\,.\\[-7mm]
	\end{align*}
\end{lemma}\newpage

\begin{proof}
	Due to \cite[Thm. 5.2]{john}, $\mathcal{B}_{\textup{Bog}}\colon C^\infty_{0,0}(\Omega)\to (C^\infty_0(\Omega))^d$ uniquely extends to a continuous operator from $L^p_0(\Omega)$ to $\smash{(W^{1,p}_0(\Omega))^d} $ 
	and, owing to $\mu_\bfA\in A_2(\mathbb{R}^d)$, also to a continuous operator from $L^2_0(\Omega;\mu_\bfA)$ to $(W^{1,2}_0(\Omega;\mu_\bfA ))^d$. In other words, there exist constants $c_1>0$, depending only on $\Omega$ and $p$, and $c_2>0$, depending only on $\Omega$ and the $[\mu_\bfA]_{A_2(\mathbb{R}^d)}$,
	such that for every $z_1\in L^p_0(\Omega)$~and~${z_2\in L^2_0(\Omega;\mu_\bfA)}$, it holds that
	\begin{align}
		\|\nabla\mathcal{B}_{\textup{Bog}}z_1\|_{p,\Omega}&\leq c_1\, \|z_1\|_{p,\Omega}\,,\label{lem:key.1}\\
		\|\nabla\mathcal{B}_{\textup{Bog}}z_2\|_{2,\mu_\bfA,\Omega}&\leq c_2\, \|z_2\|_{2,\mu_\bfA,\Omega}\,.\label{lem:key.2}
	\end{align}
	Due to $p\ge 2$, for every $t\ge 0$ and a.e.\ $x\in \Omega$, it holds that
	\begin{align}
		\begin{aligned}
		\smash{(\delta + \abs{\bfA(x)}+ t)^{p-2} t^2\leq 2^{p-2}\,( \mu_\bfA(x) \,t^2 +\, t^p)\leq 2^{p-2}\, (\delta + \abs{\bfA(x)} + t)^{p-2} t^2\,,}
	\end{aligned}\label{lem:key.3}
	\end{align}
	which, appealing to $\varphi_{\smash{\vert\bfA(x)\vert}}(t)\sim (\delta + \abs{\bfA(x)}+ t)^{p-2} t^2$ uniformly with respect to all $t\ge 0$ and a.e.\ $x\in \Omega$,
	proves that $L^p_0(\Omega)\cap L^2_0(\Omega;\mu_\bfA )=L^{\varphi_{\smash{\vert\bfA\vert}}}_0(\Omega)$. 
	Eventually,  combining  \eqref{lem:key.1}--\eqref{lem:key.3},
	for~every~$z\in L^{\varphi_{\smash{\vert\bfA\vert}}}_0(\Omega)$, 
	we conclude that
	\begin{align*}
		\rho_{\varphi_{\smash{\vert\bfA\vert}},\Omega}(\nabla\mathcal{B}_{\textup{Bog}}z)
		&\leq c\,\big(\|\nabla\mathcal{B}_{\textup{Bog}}z\|_{2,\mu_\bfA,\Omega}^2+\|\nabla\mathcal{B}_{\textup{Bog}}z\|_{p,\Omega}^p\big)
		\\&\leq c\,\big(c_2^2\|z\|_{2,\mu_\bfA,\Omega}^2+c_1^p\|z\|_{p,\Omega}^p\big)
		\\&
		\leq c\,\rho_{\varphi_{\smash{\vert\bfA\vert}},\Omega}(z)\,,
	\end{align*}
	which is the shifted modular estimate for the Bogovski\u{\i} operator.\enlargethispage{10mm}
\end{proof}

\begin{remark}\label{rem:phi_equals_p}
	Due to $p\ge 2$ and $\delta>0$, for every $t\ge 0$ and a.e.\ $x\in \Omega$, it holds that
	\begin{align*}
	\smash{\mu_\bfA(x)^{p-2} t^2\leq  (\delta + \abs{\bfA(x)} + t)^{p-2} t^2\leq \tfrac{p-2}{p} (\delta + \abs{\bfA\bfv(x)} + t)^p +\tfrac{2}{p}\,t^p \,,}
	\end{align*}
	which proves the embedding $L^p_0(\Omega)\hookrightarrow L^2_0(\Omega;\mu_\bfA )$. Thus, if $p\ge 2$ and $\delta>0$, in actual fact, we have that $L^{\varphi_{\smash{\vert\bfA\vert}}}_0(\Omega)=L^p_0(\Omega)\cap L^2_0(\Omega;\mu_\bfA )=L^p_0(\Omega)$, which, by duality, implies that $\smash{L^{(\varphi_{\smash{\vert\bfA\vert}})^*}_0(\Omega)=L^{p'}_0(\Omega)}$.
\end{remark}

Having Lemma \ref{lem:key} at hand, we are in the position to prove Lemma \ref{lem:convex_conjugation_ineq}.

\begin{proof}[Proof (of Lemma \ref{lem:convex_conjugation_ineq}).]
	Inasmuch as $\smash{((x,t)^\top\mapsto \varphi_{\smash{\vert\bfA(x)\vert}}(t))\colon \Omega\times \mathbb{R}_{\ge 0}\to \mathbb{R}_{\ge 0}}$ is a generalized N-function, resorting to   \cite[Prop.~IV.1.2]{ET99} as well as taking into account Remark \ref{rem:phi_equals_p},  for every $\mu\in \smash{L^{(\varphi_{\smash{\vert\bfA\vert}})^*}_0(\Omega)}=\Qo$, the following convex conjugation formula applies:
	\begin{align}\label{lem:convex_conjugation_ineq.1}
		\rho_{(\varphi_{\smash{\vert\bfA\vert}})^*,\Omega}(z)= \sup_{\tilde{z}\in L^{\varphi_{\smash{\vert\bfA\vert}}}_0(\Omega)}{[(z,\tilde{z})_\Omega- \rho_{\varphi_{\smash{\vert\bfA\vert}},\Omega}(\tilde{z})]}\,.
	\end{align}
	Therefore, since $\textup{div}\,\mathcal{B}_{\textup{Bog}}\tilde{z}=\tilde{z}$ for all $\tilde{z}\in \smash{L^{\varphi_{\smash{\vert\bfA\vert}}}_0(\Omega)}$, combining  \eqref{lem:convex_conjugation_ineq.1} and Lemma \ref{lem:key}, we conclude that
	\begin{align*}
		\rho_{(\varphi_{\smash{\vert\bfA\vert}})^*,\Omega}(z)&= \sup_{\tilde{z}\in L^{\varphi_{\smash{\vert\bfA\vert}}}_0(\Omega)}{\big[(z,\tilde{z})_\Omega- \rho_{\varphi_{\smash{\vert\bfA\vert}},\Omega}(\tilde{z})\big]}\\&=\sup_{\tilde{z}\in L^{\varphi_{\smash{\vert\bfA\vert}}}_0(\Omega)}{\big[(z,\textup{div}\,\mathcal{B}_{\textup{Bog}}\tilde{z})_\Omega- \rho_{\varphi_{\smash{\vert\bfA\vert}},\Omega}(\tilde{z})\big]}
		\\&\leq
		\sup_{\tilde{z}\in L^{\varphi_{\smash{\vert\bfA\vert}}}_0(\Omega)}{\big[(z,\textup{div}\,\mathcal{B}_{\textup{Bog}}\tilde{z})_\Omega-\smash{\tfrac{1}{c}}\,\rho_{\varphi_{\smash{\vert\bfA\vert}},\Omega}(\nabla\mathcal{B}_{\textup{Bog}}\tilde{z})\big]}
		\\&\leq \sup_{\bfz\in (W^{1,\varphi_{\smash{\vert\bfA\vert}}}_0(\Omega))^d}{[(z,\textup{div}\,\bfz)_\Omega- \smash{\tfrac{1}{c}}\,\rho_{\varphi_{\smash{\vert\bfA\vert}},\Omega}(\nabla\bfz)]}\,,
	\end{align*}
	which, due to $\smash{(W^{1,\varphi_{\smash{\vert\bfA\vert}}}_0(\Omega))^d=\Vo}$, is the claimed convex conjugation inequality.
\end{proof}

Eventually, having Lemma \ref{lem:shifted_modular_Pi_div} and Lemma \ref{lem:convex_conjugation_ineq} at hand, we are in the position to prove Lemma \ref{lem:discrete_convex_conjugation_ineq}.

\begin{proof}[Proof (of Lemma \ref{lem:discrete_convex_conjugation_ineq}).]  
	Using Lemma \ref{lem:convex_conjugation_ineq}, that $(z_h
        ,\Divhk\PiDG\bfz)_\Omega=(z_h ,\divo\bfz)_\Omega$ for all
        $\bfz\in \Vo$, $z_h \in \Qhkco$ (cf.~\cite[(2.31)]{kr-pnse-ldg-1}),
	and Lemma \ref{lem:shifted_modular_Pi_div}, we find that
	\begin{align*}
		\rho_{(\varphi_{\smash{\vert\bfA\vert}})^*,\Omega}(z_h )&\leq \sup_{\bfz\in \Vo}{\big[(z_h ,\textup{div}\,\bfz)_\Omega-\smash{\tfrac{1}{c}}\,\rho_{\varphi_{\smash{\vert\bfA\vert}},\Omega}(\nabla \bfz)\big]}
		\\&= \sup_{\bfz\in \Vo}{[(z_h ,\Divhk\PiDG\bfz)_\Omega-\smash{\tfrac{1}{c}}\,\rho_{\varphi_{\smash{\vert\bfA\vert}},\Omega}(\nabla \bfz)]}
		\\&\leq \sup_{\bfz\in \Vo}{[(z_h ,\Divhk\PiDG\bfz)_\Omega-\smash{\tfrac{1}{c}}\,M_{\varphi_{\smash{\vert\bfA\vert}},h}(\PiDG\bfz)]}+c\,h^2\,\|\nabla \bfF(\bfA)\|_{2,\Omega}^2
		\\&\leq \sup_{\bfz_h\in \Vhk}{[(z_h ,\Divhk\bfz_h)_\Omega-\smash{\tfrac{1}{c}}\,M_{\varphi_{\smash{\vert\bfA\vert}},h}(\bfz_h)]}+c\,h^2\,\|\nabla \bfF(\bfA)\|_{2,\Omega}^2\,,
	\end{align*}
	which is the claimed discrete convex conjugation formula.
\end{proof}

Before we can move on to the proof of Theorem \ref{thm:error_pressure}, we need to prove the following stability property result for the lifting operators $\{\pazobfcal{R}_h^k\}_{h>0}$ in terms of the shifted modular $\rho_{(\phi_{\smash{\vert\bfA\vert }})^*,\Omega}$.


\begin{lemma}\label{lem:shifted_R_stability}
	Let   $p\in [2,\infty)$ and $\delta\ge 0$. Moreover, 
	let $\bfA\in (L^p(\Omega))^{d\times d}\cap (W^{1,1}(\Omega))^{d\times d}$ with $\bfF(\bfA)\in (W^{1,2}(\Omega))^{d\times d}$. Then, there exists a constant $c>0$, depending on $k$, $\omega_0$, $p$, and $\delta$, such that for every $\gamma \in \Gamma_h$
        and $\bfw_h\in (W^{1,p}(\omega_\gamma))^d$, it holds that
	\begin{align}\label{lem:shifted_R_stability.1}
	\rho_{\varphi_{\smash{\vert\bfA\vert}},\omega_\gamma}(\pazobfcal{R}_h^k\bfw_h)\leq c\, \rho_{\phi_{\vert\bfA\vert},\gamma} (h_\gamma^{-1}\jump{\bfw_h\otimes
			\bfn})+c\, h_\gamma^2\,\|\nabla\bfF(\bfA)\|_{2,\omega_\gamma}^2\,.
	\end{align}
	In particular, there exists a constant $c>0$, depending on $k$, $\omega_0$, $p$, and $\delta$, such that for every ${\bfw_h\in (W^{1,p}(\mathcal{T}_h))^d}$, it holds that
	\begin{align}	\label{lem:shifted_R_stability.2}
	\rho_{\varphi_{\smash{\vert\bfA\vert}},\Omega}(\pazobfcal{R}_h^k\bfw_h)\leq c\,M_{\varphi_{\smash{\vert\bfA\vert}},h}(\bfw_h)+c\,h^2\,\|\nabla \bfF(\bfA)\|_{2,\Omega}^2\,.
	\end{align}
\end{lemma}

\begin{proof}
	\textit{ad \eqref{lem:shifted_R_stability.1}.}  The shift change 
	\eqref{lem:shift-change.1}, the (local) stability properties of
	$\pazobfcal{R}^k_{h,\gamma}$ (\textit{cf.}\ \cite[Lem.~A.1]{kr-phi-ldg}),
	again, the shift change 
	\eqref{lem:shift-change.1}, \cite[Lem.~A.4]{bdr-phi-stokes}, and Poincar\'e's inequality yield that 
	\begin{align}
		\label{eq:e4.3}     
		\begin{aligned}
			\rho_{\phi_{\vert{\bfA\vert }},\omega_\gamma} ( \pazobfcal{R}^k_{h,\gamma}\bfw_h)
			&\leq c\,\rho_{\phi_{\vert{\langle \bfA\rangle_{\omega_\gamma}\vert }},\omega_\gamma} ( \pazobfcal{R}^k_{h,\gamma}\bfw_h)+c\,\|\bfF(\bfA)-\bfF(\langle \bfA\rangle_{\omega_\gamma})\|_{2,\omega_\gamma}^2
			\\&\leq c\, h_\gamma\,\rho_{\phi_{\vert{\langle \bfA\rangle_{\omega_\gamma}\vert }},\gamma} (h_\gamma^{-1}\jump{\bfw_h\otimes
				\bfn})+c\,\|\bfF(\bfA)-\bfF(\langle\bfA\rangle_{\omega_\gamma})\|_{2,\omega_\gamma}^2
			\\&\leq c\, h_\gamma\,\rho_{\phi_{\vert\bfA\vert},\gamma} (h_\gamma^{-1}\jump{\bfw_h\otimes
				\bfn})+c\,\|\bfF(\bfA)-\langle\bfF( \bfA)\rangle_{\omega_\gamma}\|_{2,\omega_\gamma}^2
			\\&\leq c\, h_\gamma\,\rho_{\phi_{\vert\bfA\vert},\gamma} (h_\gamma^{-1}\jump{\bfw_h\otimes
				\bfn})+c\,h_\gamma^2\,\|\nabla\bfF(\bfA)\|_{2,\omega_\gamma}^2\,,
		\end{aligned}
	\end{align}
	which is the claimed local stability estimate
        \eqref{lem:shifted_R_stability.1} for the lifting operator $
        \pazobfcal{R}^k_{h,\gamma}$ in terms of the shifted modular $\rho_{\varphi_{\smash{\vert \bfA\vert}},\omega_\gamma}$.
	
	\textit{ad \eqref{lem:shifted_R_stability.2}.} 	Owing to
	$\pazobfcal{R}_h^k\bfz_h =\sum_{\gamma \in
		\Gamma_h}\pazobfcal{R}^k_{h,\gamma}{\bfz_h}$,
	$\textup{supp}\,\pazobfcal{R}^k_{h,\gamma}{\bfz_h}\subseteq
	\omega_\gamma$ for all
	$\gamma\in \Gamma_h$,
	$\Omega =\bigcup_{\gamma\in
		\Gamma_h}{\omega_\gamma}$, and
	that for each $\gamma\in\Gamma_h $, $\omega_\gamma$ consist of at
	most two elements, via summation with respect to
        $\gamma\in\Gamma_h$ from \eqref{lem:shifted_R_stability.1}, it
        follows the claimed global stability estimate
        \eqref{lem:shifted_R_stability.2} for the lifting operator $
        \pazobfcal{R}^k_{h,\gamma}$ in terms of the shifted modular $\rho_{\varphi_{\smash{\vert \bfA\vert}},\Omega}$.
\end{proof}

Eventually, we give a proof of a global stability result for a locally $L^1$-stable projection operator in terms of the shifted modular $\rho_{\varphi_{\vert \bfD\bfv\vert},\Omega}$.

\begin{lemma}\label{lem:shifted_Q_approx} Let $p\in [2,\infty)$, $\delta\ge 0$, and $k\in \mathbb{N}_0$. Moreover, let  $\bfA\in (L^p(\Omega))^{d\times d}$ with  $\mu_\bfA \coloneqq (\delta+\vert\overline{\bfA}\vert)^{p-2}\in A_2(\mathbb{R}^d)$ and $\bfF(\bfA)\in (W^{1,2}(\Omega))^{d\times d}$.
	If 	$\Pi_h^Q\colon Q \to \Qhk$ is a linear projection operator, which is \textup{locally $L^1$-stable}, \textit{i.e.}, for every $z\in Q$ and $K\in \mathcal{T}_h$, it holds that
	\begin{align}
		\label{eq:PiYstab}
		\langle \vert\Pi_h^Q z\vert\rangle_K  \lesssim  \langle\vert z\vert\rangle_{\omega_K}\,,
	\end{align}
	where $\omega_K\coloneqq \bigcup\{K'\in \mathcal{T}_h\mid K'\cap K\neq \emptyset\}$.
	Then, there exists a constant $c>0$, depending only on $p$, $k$, $\Omega$, $\omega_0$, and $[\mu_\bfA]_{A_2(\mathbb{R}^d)}$, such that for every $z\in \Qo$, it holds that
	\begin{align*}
		\rho_{(\phi_{\smash{\vert\bfA\vert }})^*,\Omega}(z-\Pi_h^Qz-\langle \Pi_h^Q z\rangle_\Omega)\leq \begin{cases}
			c\,h^{p'}\rho_{\phi^*,\Omega}(\nabla z)&\text{ if }z\in W^{1,p'}(\Omega)\,,\\
			c\,h^2\,\|\nabla z\|_{\smash{\mu_\bfA^{-1}},2,\Omega}^2&\text{ if } \nabla z\in (L^2(\Omega;\smash{\mu_\bfA^{-1}}))^d\,.
		\end{cases}
	\end{align*}
	In particular, we have that
	\begin{align*}
		\inf_{z_h\in \Qo^k_h}{	\rho_{(\phi_{\smash{\vert\bfA\vert }})^*,\Omega}(z-z_h)}\leq \begin{cases}
			c\,h^{p'}\rho_{\phi^*,\Omega}(\nabla z)&\text{ if }z\in W^{1,p'}(\Omega)\,,\\
			c\,h^2\,\|\nabla z\|_{\smash{\mu_\bfA^{-1}},2,\Omega}^2&\text{ if }  \nabla z\in (L^2(\Omega;\smash{\mu_\bfA^{-1}}))^d\,.
		\end{cases}
	\end{align*}
\end{lemma}

\begin{proof}\let\qed\relax
	Due to $\langle z\rangle_\Omega=0$ and the convexity of $(\phi_{\smash{\vert\bfA(x)\vert }})^*\colon \mathbb{R}^{\ge 0}\to \mathbb{R}^{\ge 0}$ for a.e.\ $x\in \Omega$, we have that\enlargethispage{7mm}
	\begin{align}\label{lem:shifted_Q_approx.1}
		\begin{aligned}
			\rho_{(\phi_{\smash{\vert\bfA\vert }})^*,\Omega}(z-\Pi_h^Qz-\langle \Pi_h^Q z\rangle_\Omega)&=	\rho_{(\phi_{\smash{\vert\bfA \vert }})^*,\Omega}(z-\Pi_h^Qz-\langle z-\Pi_h^Q z\rangle_\Omega)
			\\[-0.5mm]&\leq c\,\rho_{(\phi_{\smash{\vert\bfA\vert }})^*,\Omega}(z-\Pi_h^Qz)+c\,\rho_{(\phi_{\smash{\vert\bfA\vert }})^*,\Omega}(\langle z-\Pi_h^Q z\rangle_\Omega)\,.
		\end{aligned}
	\end{align}
	Note that, for a.e.\ $x\in \Omega$, there holds
	\begin{align}
          \begin{aligned}
            \vert \langle z-\Pi_h^Q z\rangle_\Omega\vert& \leq \vert B_{\textup{diam}(\Omega)}^d(x)\vert\vert\Omega\vert^{-1} \langle \vert (z-\Pi_h^Q z)\chi_{\Omega}\vert\rangle_{\smash{B_{\textup{diam}(\Omega)}^d(x)}}
            \\[-1mm]&\leq c\,M_d( (z-\Pi_h^Q z)\chi_\Omega)(x)\,,
          \end{aligned}
	\end{align}
	where  $M_d\colon L^1_{\textup{loc}}(\Omega)\to L^0(\Omega)$ denotes the Hardy--Littlewood maximal operator.
	 
	Next, we need to distinguish the cases $z\in W^{1,p'}(\Omega)$ and $\nabla z\in (L^2(\Omega;\smash{\mu_\bfA^{-1}}))^d$:
	
	\textit{Case $z\in W^{1,p'}(\Omega)$.} In this case, using that $(\phi_{\smash{\vert\bfA(x)\vert }})^*(t)\leq c\, \phi^*(t)$ for all $t\ge 0$ and a.e.\ $x\in \Omega$, the stability of $M_d$ from $L^{\varphi^*}(\mathbb{R}^d)$ to $L^{\varphi^*}(\mathbb{R}^d)$, and \cite[Lem.\ 5.3]{bdr-phi-stokes}, we find that
	\begin{align*}
		\rho_{(\phi_{\smash{\vert\bfA\vert }})^*,\Omega}(\langle (z-\Pi_h^Q z) \rangle_\Omega\,\chi_{\Omega})&\leq 	\rho_{\phi^*,\mathbb{R}^d}(M_d( (z-\Pi_h^Q z)\chi_{\Omega}))
		\\[-0.5mm]&\leq c\,	\rho_{\phi^*,\Omega}(z-\Pi_h^Q z)
		\\[-0.5mm]&\leq c\,h^{p'}\,\rho_{\phi^*,\Omega}(\nabla z)\,.
	\end{align*}
	
	\textit{Case $\nabla z\hspace*{-0.1em}\in\hspace*{-0.1em}
          (L^2(\Omega;\smash{\mu_\bfA^{-1}}))^d$.} In this case, using
        that $(\phi_{\smash{\vert\bfA(x)\vert
          }})^*(t)\hspace*{-0.1em}\leq
        \hspace*{-0.1em}c\,\mu_\bfA(x)^{-1}\,t^2$~for~all~${t\hspace*{-0.1em}\ge\hspace*{-0.1em}
          0}$~and~a.e.~${x\hspace*{-0.1em}\in\hspace*{-0.1em}
          \Omega}$, the stability of $M_d$ from
        $L^2(\mathbb{R}^d;\smash{\mu_\bfA^{-1}})$ to
        $L^2(\mathbb{R}^d;\smash{\mu_\bfA^{-1}})$ (\textit{cf.}\
       \cite{singular_integrals}),  and
        Lemma~\ref{lem:PiYstab_muckenhoupt}\eqref{lem:PiYstab_muckenhoupt.2}, we find that
	\begin{align*}
		\rho_{(\phi_{\smash{\vert\bfA\vert }})^*,\Omega}(\langle (z-\Pi_h^Q z) \rangle_\Omega\,\chi_{\Omega})&\leq c\,	\|M_d( (z-\Pi_h^Q z)\chi_{\Omega})\|_{2,\smash{\mu_\bfA^{-1}},\mathbb{R}^d}^2
		\\[-0.5mm]&\leq 	c\,\|z-\Pi_h^Q z\|_{2,\smash{\mu_\bfA^{-1}},\Omega}^2
		\\&\leq 	c\,h^2\,\|\nabla z\|_{2,\smash{\mu_\bfA^{-1}},\Omega}^2\,.\tag*{$\qedsymbol$}
	\end{align*}
\end{proof}

\vspace*{-1cm}
\begin{proof}[Proof (of Theorem \ref{thm:error_pressure}).] 
	Appealing to Lemma~\ref{lem:discrete_convex_conjugation_ineq}, there exists a constant $c>0$ such
	that for every $z_h\in \Qhkco$, it holds that\enlargethispage{5mm}
	\begin{align}
		\label{thm:error.1}
		\rho_{(\varphi_{\smash{\vert\bfD\bfv\vert}})^*,\Omega}(z_h)\leq \sup_{\bfz_h\in \Vhk}{[(z_h,\Divhk\,\bfz_h)_\Omega-\smash{\tfrac{1}{c}}\,M_{\varphi_{\smash{\vert\bfD\bfv\vert}},h}(\bfz_h)]}+c\,h^2\,\|\nabla \bfF(\bfD\bfv)\|_{2,\Omega}^2\,.
	\end{align}
	Appealing to  \cite[(4.6)]{kr-pnse-ldg-3}, for every $\bfz_h\in \Vhk$, we have that
	\begin{align}\label{thm:error.2}
		\begin{aligned}
			(q_h-q,\Divhk \bfz_h)_\Omega
			&
			=(\SSS(\Dhk \bfv_h) - \SSS(\bfD
				\bfv),\Dhk \bfz_h)_\Omega
			\\&
			\quad
			+ \alpha \langle\SSS_{\smash{\sssl}}(h^{-1} \jump{(\bfv_h-\bfv)\otimes
				\bfn}), \jump{\bfz_h \otimes \bfn}\rangle_{\Gamma_h}
			\\&
			\quad
			-b_h(\bfv,\bfv,\bfz_h)+b_h(\bfv_h,\bfv_h,\bfz_h)
			\\&
			\quad
			+\langle\avg{\PiDG\bfS}-\avg{\bfS},\jump{\bfz_h\otimes
					\bfn}\rangle_{\Gamma_h}
			\\&
			\quad
			+\tfrac{1}{2}\langle\avg{\bfK}-\avg{\PiDG\bfK},\jump{\bfz_h\otimes\bfn}\rangle_{\Gamma_h}
			\\&
			\quad
			+\langle \avg{q\mathbf{I}_d}-\avg{(\PiDG q)\mathbf{I}_d},\jump{\bfz_h\otimes\bfn}\rangle_{\Gamma_h}
			\\&=I_h^1+\alpha\, I_h^2+I_h^3+\ldots+I_h^6\,,
		\end{aligned}
	\end{align}
	where the~\textit{discrete~convective~term} 
	$b_h\colon[(\WDG)^d]^3 \to  \mathbb{R}$, for every $(\bfx_h,\bfy_h,\bfz_h)^\top\in [(\WDG)^d]^3$,\linebreak is defined via
	\begin{align}
		b_h(\bfx_h,\bfy_h,\bfz_h)\coloneqq \tfrac{1}{2}( \bfz_h\otimes \bfx_h,\Ghk\bfy_h)_\Omega-\tfrac{1}{2}(\bfy_h\otimes \bfx_h,\Ghk\bfz_h)_\Omega\,.\label{def:bh}
	\end{align}
	So, let us next estimate $I_h^1,\ldots, I_h^6$: 
	
	\textit{ad $I_h^1$}.  Using the $\varepsilon$-Young inequality
        \eqref{ineq:young} with $\psi =\varphi_{\smash{\vert\bfD\bfv\vert}}$,
        Proposition \ref{lem:hammer}, \eqref{eq:phi*phi'} with $\psi =\varphi_{\smash{\vert\bfD\bfv\vert}}$, $\Ghk\bfv_h=\nabla_h\bfv_h-\pazobfcal{R}_h^k\bfv_h$ together
        with Lemma \ref{lem:shifted_R_stability}  \eqref{lem:shifted_R_stability.2}, and Theorem
        \ref{thm:error_LDG}, we find that
	\begin{align}\label{thm:error.3}
		\begin{aligned}
			\vert I_h^1\vert &\leq c_\varepsilon\,\rho_{(\varphi_{\smash{\vert\bfD\bfv\vert}})^*,\Omega}(\bfS(\Dhk\bfv_h)-\bfS(\bfD\bfv))+\varepsilon\,\rho_{\varphi_{\smash{\vert\bfD\bfv\vert}},\Omega}(\Dhk\bfz_h)\\&\leq
			c_\varepsilon\,\|
					\bfF(\Dhk\bfv_h)-\bfF(\bfD\bfv)\|_{2,\Omega}^2+\varepsilon\,c\,\big(M_{\varphi_{\smash{\vert\bfD\bfv\vert}},h}(\bfz_h)+h^2 \,\|\nabla\bfF(\bfD\bfv)\|_{2,\Omega}^2\big)
			\\&\leq
			c_\varepsilon\,\big(h^2 \,\|\nabla \bfF(\bfD\bfv)\|_{2,\Omega}^2
			+c\,\rho_{(\varphi_{\vert \bfD\bfv\vert})^*,\Omega}(h\,\nabla q)\big)+\varepsilon\,c\,M_{\varphi_{\smash{\vert\bfD\bfv\vert}},h}(\bfz_h)
			\,.
		\end{aligned}
	\end{align}

	\textit{ad $I_h^2$}.
	Using the $\varepsilon$-Young inequality \eqref{ineq:young} with $\psi =\varphi_{\smash{\avg{\vert\Pia
				\Dhk\bfv_h\vert}}}$, Proposition \ref{lem:hammer}, \eqref{eq:phi*phi'} with $\psi =\varphi_{\smash{\avg{\vert\Pia
				\Dhk\bfv_h\vert}}}$, the shift change \eqref{lem:shift-change.1}, \cite[Lem.\ 4.12]{kr-pnse-ldg-2}, and Theorem \ref{thm:error_LDG}, we find that
	\begin{align}\label{thm:error.8}
	\begin{aligned}
			\vert I_h^2\vert &\leq c_\varepsilon\,h\,\rho_{(\varphi_{\smash{\avg{\vert\Pia
							\Dhk\bfv_h\vert}}})^*,\Gamma_h}\big(\bfS_{\smash{\avg{\vert\Pia
					\Dhk\bfv_h\vert}}}(h^{-1}\jump{\bfv_h\otimes\bfn})\big)+\varepsilon\,h\,\rho_{\varphi_{\smash{\avg{\vert\Pia
							\Dhk\bfv_h\vert}}},\Gamma_h}(h^{-1}\jump{\bfz_h\otimes\bfn})
			\\	&\leq c_\varepsilon\,m_{\varphi_{\smash{\avg{\vert\Pia \Dhk\bfv_h\vert}}},h}(\bfv_h)+\varepsilon\,c\,m_{\varphi_{\smash{\vert
						\bfD\bfv\vert}},h}(\bfz_h)+\varepsilon\,c\,h\,\rho_{\varphi_{\smash{\vert
						\bfD\bfv\vert}},\Gamma_h}(\vert \bfD\bfv\vert-\avg{\vert\Pia \Dhk\bfv_h\vert})
			\\	&\leq c_\varepsilon\,m_{\varphi_{\smash{\avg{\vert\Pia \Dhk\bfv_h\vert}}},h}(\bfv_h-\bfv)+\varepsilon\,c\,m_{\varphi_{\smash{\vert
						\bfD\bfv\vert}},h}(\bfz_h)
                                           \\&\quad
                                          +	\varepsilon\,c\,\big(\|
				\bfF(\Dhk\bfv_h)-\bfF(\bfD\bfv)\|_{2,\Omega}^2+h^2
          \,\|\nabla\bfF(\bfD\bfv)\|_{2,\Omega}^2\big) 
			\\	&\leq c_\varepsilon\,\big(\,h^2
                        \,\|\nabla \bfF(\bfD\bfv)\|_{2,\Omega}^2
			+c\,\rho_{(\varphi_{\vert \bfD\bfv\vert})^*,\Omega}(h\,\nabla q)\big)+\varepsilon\,c\,M_{\varphi_{\smash{\vert\bfD\bfv\vert}},h}(\bfz_h)
			\,. 
	\end{aligned}
	\end{align}
	
	\textit{ad $I_h^3$}. Introducing the notation $\bfe_h\coloneqq  \bfv_h-\bfv\in (\WDG)^d$, the term $I_h^3$ can be re-written~as
	\begin{align}\label{eq:i3}
		\begin{aligned}
			I_h^3 &= b_h(\bfv,\bfv-\PiDG\bfv,\bfz_h)-b_h(\bfe_h,\PiDG\bfv,\bfz_h)+b_h(\bfv_h,\PiDG\bfe_h,\bfz_h)
			\\&\eqqcolon  I_h^{3,1}+I_h^{3,2}+I_h^{3,3}\,.
		\end{aligned}
	\end{align}
	So, we have to estimate $ I_h^{3,i}$, $i=1,2,3$:
	
	\textit{ad $I_h^{3,1}$}.  The definition of $b_h\colon[(\WDG)^d]^3\to \mathbb{R}$ (\textit{cf.}\ \eqref{def:bh})
	yields that
	\begin{align}\label{eq:i31}
		\begin{aligned}
			2\,I_h^{3,1}=(\bfz_h\otimes \bfv,\Ghk(\bfv-\PiDG\bfv))_\Omega-((\bfv-\PiDG\bfv)\otimes \bfv,\Ghk\bfz_h)_\Omega\,,
		\end{aligned}
	\end{align}
	so that using Hölder's inequality, that
        $\|\bfv\|_{\infty,\Omega}\leq c\, \|\bfv\|_{2,2,\Omega}\leq
        c\,\delta^{2-p}\,\|\bfF(\bfD\bfv)\|_{1,2,\Omega}$ (\textit{cf.}\
        \cite[Lem.\ 4.5]{bdr-7-5}), the discrete Sobolev theorem (\textit{cf.}\ \cite[Prop.\ 2.6]{kr-pnse-ldg-3}), \eqref{eq:eqiv0}, the $\varepsilon$-Young inequality \eqref{ineq:young} with $\psi=\vert\cdot\vert^2$, the approximation properties of $\PiDG$
        (\textit{cf.}~\cite[Cor.\ A.4, Cor.\ A.7]{kr-phi-ldg}), $\|\nabla^2\bfv\|_{2,\Omega}\leq
        c_\varepsilon\,\delta^{2-p}\,\|\nabla\bfF(\bfD\bfv)\|_{2,\Omega}$\linebreak
         (\textit{cf.}\ \cite[Lem.\ 4.5]{bdr-7-5}), and
        $\|\bfz_h\|_{\nabla,2,h}^2 \le c\, \delta^{2-p}\,M_{\varphi_{\smash{\vert \bfD\bfv\vert}},h}(\bfz_h) $,~we~obtain
	\begin{align}
		\label{eq:i311}
		\begin{aligned}
			\vert I_h^{3,1}\vert &\leq 2\,\|\bfv\|_{\infty,\Omega}(\|\bfv-\PiDG\bfv\|_{2,\Omega}+\|\Ghk(\bfv-\PiDG\bfv)\|_{2,\Omega})(\|\bfz_h\|_{2,\Omega}+\|\Ghk\bfz_h\|_{2,\Omega})
			\\
			&\leq
                        c_\varepsilon\,\|\bfv-\PiDG\bfv\|_{\nabla,2,h}^2+\varepsilon\,\|\bfz_h\|_{\nabla,2,h}^2
				\\
			&\leq c_\varepsilon\,h^2\,\|\nabla^2\bfv\|_{2,\Omega}^2+\varepsilon\, \|\bfz_h\|_{\nabla,2,h}^2
				\\
			&\leq c_\varepsilon\,h^2\,\|\nabla\bfF(\bfD\bfv)\|_{2,\Omega}^2+\varepsilon\,c\, M_{\varphi_{\smash{\vert \bfD\bfv\vert}},h}(\bfz_h)\,.
		\end{aligned}
	\end{align}
	
	\textit{ad $I_h^{3,2}$}. The definition of $b_h\colon[(\WDG)^d]^3\to \mathbb{R}$ (\textit{cf.}\ \eqref{def:bh})
	yields that\enlargethispage{10mm}
	\begin{align}
		\label{eq:i32}
		\begin{aligned}
			2\,I_h^{3,2}=(\bfz_h\otimes \bfe_h,\Ghk\PiDG\bfv)_\Omega+(\PiDG\bfv\otimes \bfe_h,\Ghk\bfz_h)_\Omega\,, 
		\end{aligned}
	\end{align}
	so that using Hölder's inequality, the discrete Sobolev theorem (\textit{cf.}\ \cite[Prop.\ 2.6]{kr-pnse-ldg-3}),  \eqref{eq:eqiv0}, the DG-stability property of
	$\PiDG$ (\textit{cf.}~\cite[(A.19)]{dkrt-ldg}), the $\varepsilon$-Young inequality \eqref{ineq:young} with $\psi=\vert\cdot\vert^2$, 
	the
	Korn type inequality \cite[Prop.\ 2.8]{kr-pnse-ldg-2}, $\bfv
        \in W^{1,2}(\Omega)^d$, that $\vert \bfA-\bfB\vert^2\leq c\, \delta^{2-p}\,\vert \bfF(\bfA)-\bfF(\bfB)\vert^2$  and $\vert \bfA\vert^2\leq c\, \delta^{2-p}\,\varphi_{\vert \bfB\vert}(\vert\bfA\vert)$ for all $\bfA,\bfB\in \mathbb{R}^{d\times d}$ (due to $p>2$), that $\|\nabla^2\bfv\|_{2,\Omega}\leq c\,\delta^{2-p}\,\|\nabla \bfF(\bfD\bfv)\|_{2,\Omega}$\linebreak (\textit{cf.}\ \cite[Lem.\ 4.5]{bdr-7-5}),~and Theorem~\ref{thm:error_LDG}, we~find~that
	\begin{align}
		\label{eq:i321}
		\begin{aligned}
				\vert I_{h}^{3,2}\vert &\leq
		(\|\PiDG\bfv\|_{4,\Omega}+\|\Ghk\PiDG\bfv\|_{2,\Omega})\|\bfe_h\|_{4,\Omega}(\|\bfz_h\|_{4,\Omega}+
		\|\Ghk\bfz_h\|_{2,\Omega})
			\\
			&\leq
			c\,\|\PiDG\bfv\|_{\nabla,2,h}\|\bfe_h\|_{\nabla,2,h}\|\bfz_h\|_{\nabla,2,h}
			\\
			&\leq
			c_\varepsilon\,\|\nabla\bfv\|_{2,\Omega}^2\big (\|\bfe_h\|_{\bfD,2,h}^2+h^2\,\|\nabla ^2 \bfv\|_{2,\Omega}^2\big )+\varepsilon\,\|\bfz_h\|_{\nabla,2,h}^2
			\\
			&\leq
			c_\varepsilon\,\big (\|
					\bfF(\Dhk\bfv_h)-\bfF(\bfD\bfv)\|_{2,\Omega}^2+m_{\phi_{\smash{\sssl}},h }
			(\bfv_h-\bfv)  +h^2\,\|\nabla\bfF(\bfD
				\bfv)\|_{2,\Omega}^2\big )+\varepsilon\,c\,M_{\varphi_{\smash{\vert \bfD\bfv\vert}},h}(\bfz_h)\hspace*{-5mm}\\
			&\leq c_\varepsilon\,\big(h^2\,\|\nabla \bfF(\bfD\bfv)\|_{2,\Omega}^2+\rho_{(\varphi_{\vert \bfD\bfv\vert})^*,\Omega}(h\,\nabla q)\big)+\varepsilon\,c\,M_{\varphi_{\smash{\vert \bfD\bfv\vert}},h}(\bfz_h)\,.
		\end{aligned}
	\end{align}
	
	\textit{ad $I_h^{3,3}$}.  The definition of $b_h\colon [(\WDG)^d]^3\to \mathbb{R}$ (\textit{cf.}\ \eqref{def:bh})
	yields that
	\begin{align}
		\label{eq:i33}
		\begin{aligned}
			2\,I_h^{3,3}=(\bfz_h\otimes \bfv_h,\Ghk\PiDG\bfe_h)_\Omega+(\PiDG\bfe_h\otimes \bfv_h,\Ghk\bfz_h)_\Omega\,, 
		\end{aligned}
	\end{align}
	so that using Hölder's inequality, the discrete Sobolev theorem (\textit{cf.}\ \cite[Prop.\ 2.6]{kr-pnse-ldg-3}),  \eqref{eq:eqiv0}, the DG-stability property of
	$\PiDG$ (\textit{cf.}~\cite[(A.19)]{dkrt-ldg}), the a priori estimate
        \cite[Prop.~5.7]{kr-pnse-ldg-1}, the $\varepsilon$-Young inequality \eqref{ineq:young} with $\psi=\vert\cdot\vert^2$, 
	the
	Korn type inequality \cite[Prop.\ 2.8]{kr-pnse-ldg-2}, that $\vert \bfA-\bfB\vert^2\leq c\, \delta^{2-p}\,\vert \bfF(\bfA)-\bfF(\bfB)\vert^2$  and $\vert \bfA\vert^2\leq c\, \delta^{2-p}\,\varphi_{\vert \bfB\vert}(\vert\bfA\vert)$ for all $\bfA,\bfB\in \mathbb{R}^{d\times d}$ (due to $p>2$), that $\|\nabla^2\bfv\|_{2,\Omega}\leq c\,\delta^{2-p}\,\|\nabla \bfF(\bfD\bfv)\|_{2,\Omega}$ \linebreak(\textit{cf.}\ \cite[Lem.\ 4.5]{bdr-7-5}),~and Theorem~\ref{thm:error_LDG}, we~find~that
	\begin{align}
		\label{eq:i331}
		\begin{aligned}
			\abs{I_h^{3,3}}&\leq \|\bfv_h\|_{4,\Omega}(\|\PiDG\bfe_h\|_{4,\Omega}+\|\Ghk\PiDG\bfe_h\|_{2,\Omega})(\|\bfz_h\|_{4,\Omega}+\|\Ghk\bfz_h\|_{2,\Omega})\\
			&\leq
			c\,\norm{\bfv_h}_{\nabla,p,h}\norm{\PiDG\bfe_h}_{\nabla,2,h}\|\bfz_h\|_{\nabla,2,h}\\
			&\leq
			c_\varepsilon\,\|\bfe_h\|_{\nabla,2,h}^2+\varepsilon\,\|\bfz_h\|_{\nabla,2,h}^2\\
			&\leq c_\varepsilon\,\big(h^2\,\|\nabla \bfF(\bfD\bfv)\|_{2,\Omega}^2+\rho_{(\varphi_{\vert \bfD\bfv\vert})^*,\Omega}(h\,\nabla q)\big)+\varepsilon\,c\,M_{\varphi_{\smash{\vert \bfD\bfv\vert}},h}(\bfz_h)\,.
		\end{aligned}
	\end{align}
	
	Eventually, combining \eqref{eq:i3}--\eqref{eq:i331},  we conclude that
	\begin{align}\label{eq:i3fin}
		\begin{aligned}
			\vert I_h^3\vert \leq  c_\varepsilon  \,\big(h^2\,\|\nabla \bfF(\bfD\bfv)\|_{2,\Omega}^2
			+\rho_{(\varphi_{\vert \bfD\bfv\vert})^*,\Omega}(h\,\nabla q)\big)+\varepsilon\,c\,M_{\varphi_{\vert \bfD\bfv\vert},h}(\bfz_h)\,.
		\end{aligned}
	\end{align}
	
	\textit{ad $I_h^4$}.
	Using the $\varepsilon$-Young inequality \eqref{ineq:young}
        with $\psi=\varphi_{\smash{\vert \bfD\bfv\vert}}$  and 
        \cite[(4.43), (4.45)]{kr-pnse-ldg-2},  
        we find that
	\begin{align}\label{thm:error.14}
		\begin{aligned}
			\vert I_h^4\vert &\leq c_\varepsilon\,	h\,\rho_{(\varphi_{\vert \bfD\bfv\vert})^*,\Gamma_h}(\avg{\PiDG\bfS(\bfD\bfv)}-\avg{\bfS(\bfD\bfv)})
			+	\varepsilon\,h\,\rho_{\varphi_{\vert
					\bfD\bfv\vert},\Gamma_h}(h^{-1}\jump{\bfz_h\otimes\bfn})
			\\&\leq c_\varepsilon\,
			h^2 \,\|\nabla \bfF(\bfD\bfv)\|_{2,\Omega}^2 
			+	\varepsilon\,c\,M_{\varphi_{\vert
					\bfD\bfv\vert},h}(\bfz_h)\,.
		\end{aligned}
	\end{align}

	\textit{ad $I_h^5$}.
	Using the $\varepsilon$-Young inequality \eqref{ineq:young} with $\psi=\vert\cdot\vert^2$ and, taking into account 
	that  $\bfK =\bfv \otimes \bfv\in (W^{1,2}_0(\Omega))^{d\times d}$, where $\|\nabla     \bfK\|_{2,\Omega} \le c\,
	\|\bfv\|_{\infty,\Omega}\,\delta^{2-p}\|\bfF(\bfD\bfv)\|_{2,\Omega} $, using \cite[Cor. A.7]{kr-phi-ldg}, we find that
	\begin{align}\label{thm:error.19}
		\begin{aligned} 
		\vert I_h^5\vert
		&\leq  c_\varepsilon\, h\,\norm{\avg{\bfK}-\avg{\PiDG
				\bfK}}_{2,\Gamma_h}^2+\varepsilon\, h\,\norm{h^{-1}\jump{\bfz_h\otimes\bfn}}_{2,\Gamma_h}^2
			\\ &\leq  c_\varepsilon\, h^2\,\|\bfF(\bfD\bfv)\|_{2,\Omega}^2+\varepsilon\, M_{\varphi_{\vert
				\bfD\bfv\vert},h}(\bfz_h)\,.
			\end{aligned}
	\end{align}
	
	\textit{ad $I_h^6$}.
	Using the $\varepsilon$-Young inequality \eqref{ineq:young}
        with $\psi=\varphi_{\smash{\vert \bfD\bfv\vert}}$ and
        \cite[Prop.~4.10]{kr-pnse-ldg-2}, we find that
	\begin{align}\label{thm:error.19.1}
		\begin{aligned}
		\vert I_h^6\vert &\leq c_\varepsilon\,h\,	\rho_{(\varphi_{\vert
				\bfD\bfv\vert})^*,\Gamma_h}(\avg{q\mathbf{I}_d}-\avg{\PiDG(q\mathbf{I}_d)})	+\varepsilon\,	h\,\rho_{\varphi_{\vert
				\bfD\bfv\vert},h}(h^{-1}\jump{\bfz_h\otimes\bfn})
		\\&\leq c_\varepsilon\,\big(h^2 \,\|\nabla \bfF(\bfD\bfv)\|_{2,\Omega}^2
		+\rho_{(\varphi_{\vert \bfD\bfv\vert})^*,\Omega}(h\,\nabla q)  \big)
		+\varepsilon\,	M_{\varphi_{\vert
				\bfD\bfv\vert},h}(\bfz_h)\,.
			\end{aligned}
	\end{align}
	
		Putting everything together, from \eqref{thm:error.3}, \eqref{thm:error.8}, \eqref{eq:i3fin}, \eqref{thm:error.14}, \eqref{thm:error.19},  and \eqref{thm:error.19.1} 
	in \eqref{thm:error.2}, for every $\bfz_h\in \Vhk$, we conclude that
	\begin{align}\label{thm:fast.0}
		(q_h-q,\Divhk \bfz_h)_\Omega\leq  c_\varepsilon\,(h^2\,\|\nabla\bfF(\bfD\bfv)\|_{2,\Omega}^2+\rho_{(\phi_{\smash{\vert\bfD \bfv\vert }})^*,\Omega}(h\,\nabla q))+\varepsilon\, c\,M_{\varphi_{\smash{\vert
					\bfD\bfv\vert}},h}(\bfz_h)\,.
	\end{align}
        Using Lemma \ref{lem:discrete_convex_conjugation_ineq} and
        \eqref{thm:fast.0} for sufficiently small $\varepsilon>0$, the $\varepsilon$-Young inequality \eqref{ineq:young}
        with $\psi=\varphi_{\smash{\vert \bfD\bfv\vert}} $ for sufficiently small $\varepsilon>0$, for every $z_h\in \Qhkco$, we find  that
	\begin{align}\label{eq:fast.1}
		\begin{aligned}
                  \rho_{(\varphi_{\smash{\vert
                        \bfD\bfv\vert}})^*,\Omega}(q_h-q) &\leq
                  c\,\rho_{(\varphi_{\smash{\vert
                        \bfD\bfv\vert}})^*,\Omega}(q_h-z_h)+c\,
                  \rho_{(\varphi_{\smash{\vert
                        \bfD\bfv\vert}})^*,\Omega}(q-z_h)
                  \\
                  &\leq c\,
                  \sup_{\bfz_h\in
                    \Vo_h}{[(q_h-z_h,\Divhk\bfz_h)_\Omega-\smash{\tfrac{1}{c}}\,M_{\varphi_{\smash{\vert
                          \bfD\bfv\vert}},h}(\bfz_h)]}
                  \\
                  &\quad +c\,h^2 \,\|\nabla\bfF(\bfD\bfv)\|_{2,\Omega}^2+c\,\rho_{(\varphi_{\smash{\vert
                        \bfD\bfv\vert}})^*,\Omega}(q-z_h)
                  \\
                  &\leq c\,
                  \sup_{\bfz_h\in
                    \Vo_h}{[(q-z_h,\Divhk\bfz_h)_\Omega-\smash{\tfrac{1}{c}}\,M_{\varphi_{\smash{\vert
                          \bfD\bfv\vert}},h}(\bfz_h)]}+c\,\rho_{(\phi_{\smash{\vert\bfD \bfv\vert }})^*,\Omega}(h\,\nabla q))
                      \\&\quad +c\,h^2 \,\|\nabla\bfF(\bfD\bfv)\|_{2,\Omega}^2+c\,\rho_{(\varphi_{\smash{\vert
                        \bfD\bfv\vert}})^*,\Omega}(q-z_h) \\&\leq
                  c\,h^2 \,\|\nabla\bfF(\bfD\bfv)\|_{2,\Omega}^2
                  +c\,\rho_{(\varphi_{\smash{\vert
                        \bfD\bfv\vert}})^*,\Omega}(h\,\nabla
                  q)+c\,\rho_{(\varphi_{\smash{\vert
                        \bfD\bfv\vert}})^*,\Omega}(q-z_h)\,.
		\end{aligned}
	\end{align}
	Eventually, taking in \eqref{eq:fast.1} the infimum with respect to $z_h\in \Qhkco$, we conclude the claimed a priori error estimate for the pressure. 
\end{proof}

\begin{proof}[Proof (of Corollary \ref{cor:error_pressure}).] 
	Follows from Theorem \ref{thm:error_pressure} by means of Lemma~\ref{lem:shifted_Q_approx}.
\end{proof}

\newpage

\section{Numerical experiments}\label{sec:experiments}

\hspace*{5mm}In this section, we complement the theoretical findings of Section \ref{sec:ldg} with numerical experiments:~first, we carry out numerical experiments to confirm the quasi-optimality of the a priori error estimates in Corollary \ref{cor:error_pressure} with respect to the Muckenhoupt regularity condition~\eqref{eq:reg-assumption}; second, we carry out numerical experiments to examine the Muckenhoupt regularity condition \eqref{eq:reg-assumption} for its sharpness.~Before we do so, we first give some implementation details.

\subsection{Implementation details}
\hspace*{5mm}All experiments were carried out using the finite element software~\mbox{\texttt{FEniCS}} (version 2019.1.0, \textit{cf.}~\cite{LW10}). 
All graphics were generated with the help of the \texttt{Matplotlib} library (version 3.5.1,~\textit{cf.}~\cite{Hun07}). 
In the numerical experiments using the LDG formulation Problem (\hyperlink{Qhldg}{Q$_h$}) (or Problem (\hyperlink{Phldg}{P$_h$}), respectively), we always deploy element-wise affine, \textit{i.e.}, $k=1$, ansatz functions.

We approximate discrete solutions $(\bfv_h,q_h)^{\top}\in V_h^k\times \Qhkco$ of Problem (\hyperlink{Qhldg}{Q$_h$}) using the Newton solver from \texttt{PETSc} (version~3.17.3,  \textit{cf.}\ \cite{LW10}), with absolute tolerance of $\tau_{abs}= 1.0\times 10^{-8}$ and relative tolerance of $\tau_{rel}=1.0\times 10^{-10}$. The linear system emerging in each Newton iteration is solved using a sparse direct solver from \texttt{MUMPS} (version~5.5.0,~\textit{cf.}~\cite{mumps}).  In the implementation, the uniqueness of the pressure is enforced via adding a zero mean condition. Each integral that does contain non-discrete functions (\textit{e.g.}, the right-hand side integral in  Problem (\hyperlink{Qhldg}{Q$_h$}) and Problem (\hyperlink{Phldg}{P$_h$}) or the error quantities \eqref{error_quantities}) is discretized using a quadrature rule by G.\ Strang and G.\ Fix (\textit{cf.}\ \cite{strang_fix}) with degree of precision 6 (\textit{i.e.}, 12 quadrature points on each triangle) in two dimensions (\textit{i.e.}, $d=2$) and  using
the Keast rule \texttt{(KEAST7)} (\textit{cf.}\ \cite{keast}) with degree of precision 6 (\textit{i.e.}, employing
24 quadrature points on each element) in three dimensions (\textit{i.e.}, $d=3$).
\subsubsection{\itshape General experimental set-up}

\hspace*{5mm}We always employ the extra stress tensor $\bfS\colon  \mathbb{R}^{d\times d}\to\smash{\mathbb{R}^{d\times d}_{\mathrm{sym}}}$, where $\Omega=(0,1)^d$, for  every $\bfA\in\mathbb{R}^{d\times d}$ defined via
\begin{align*}
	\bfS(\bfA) \coloneqq \mu_0\,(\delta+\vert \bfA^{\textup{sym}}\vert)^{p-2}\bfA^{\textup{sym}}\,,
\end{align*}  
where  $\delta\coloneqq 1.0\times 10^{-5}$, $\mu_0 \coloneqq \frac{1}{2}$, and $ p\in [2.25, 3.5]$, which satisfies Assumption \ref{assum:extra_stress}.

We construct an initial triangulation $\pazocal{T}_{h_0}$, where $h_0=1$, for $d\in \{2,3\}$, in the following way:

\begin{itemize}[{($d=3$)}]
	\item[($d=2$)] We subdivide $\Omega=(0,1)^2$ into four triangles along its diagonals.
	\item[($d=3$)] We subdivide $\Omega=(0,1)^3$ into six tetrahedron forming a Kuhn triangulation $\mathcal{T}_0$ (\textit{cf.}\ \cite{kuhn}).
\end{itemize}
Then, finer triangulations $\pazocal{T}_{h_i}$, $i=1,\ldots,6$, where $h_{i+1}=\frac{h_i}{2}$ for all $i=0,\ldots,5$, are 
obtained by  applying the red-refinement rule (\textit{cf.}\ \cite[Def.~4.8(i)]{Ba16}). 

Then, for $i \hspace*{-0.1em}=\hspace*{-0.1em} 0,\ldots,6$, we compute discrete solutions  $(\bfv_i,q_i)^{\top}\hspace*{-0.15em}\coloneqq\hspace*{-0.15em}(\bfv_{h_i},q_{h_i})^{\top}\hspace*{-0.2em}\in\hspace*{-0.15em} \smash{V_{h_i}^k\hspace*{-0.15em}\times\hspace*{-0.15em} \Qo_{h_i,c}^k}$~of~Problem~(\hyperlink{Qhldg}{Q$_{h_i}$}) 
and the error quantities
\begin{align}\label{error_quantities} 
	\left.\begin{aligned}
		e_{\bfF,i}&\coloneqq \|\bfF(\smash{\pazobfcal{D}_{h_i}^1}\bfv_i)-\bfF(\bfD\bfv)\|_{2,\Omega}\,,\\[0.5mm]
		e_{\jump{},i}&\coloneqq (m_{\varphi_{\smash{\avg{\vert\Pi_{h_i}^0 \pazobfcal{D}_{h_i}^1\bfv_i\vert}}},h}(\bfv_i-\bfv))^{\smash{1/2}}\,,\\[1mm]
		e_{q,i}^{\textrm{norm}}&\coloneqq \|q_i-q\|_{p',\Omega}\,,\\
		e_{q,i}^{\textrm{modular}}&\coloneqq (\rho_{(\varphi_{\smash{\vert\bfD\bfv\vert}})^*,\Omega}(q_i-q))^{\smash{1/2}}\,,
	\end{aligned}\quad\right\}\,,\quad i=0,\ldots,6\,.
\end{align}

As  estimation  of  the  convergence rates, we compute the experimental order of convergence~(EOC)
\begin{align}
	\texttt{EOC}_i(e_i)\coloneqq\frac{\log(e_i/e_{i-1})}{\log(h_i/h_{i-1})}\,, \quad i=1,\ldots,6\,,\label{eoc}
\end{align}
where for every $i= 0,\ldots,6$, we denote by $e_i$ either $\smash{e_{\bfF,i}}$, $\smash{e_{\jump{},i}}$, $ \smash{e_{q,i}^{\textrm{norm}}}$, or $\smash{e_{q,i}^{\textrm{modular}}}$.

\subsubsection{\itshape Quasi-optimality of the a priori error estimates derived in Corollary \ref{cor:error_pressure}}

\hspace*{5mm}In order to confirm the quasi-optimality of the a priori error estimates derived in Corollary \ref{cor:error_pressure},  we restrict to the two-dimensional case and choose the right-hand side $\bff \in (L^{p'}(\Omega))^2$ and the Dirichlet boundary data $\bfv_0\in (W^{1,1-1/p}(\partial\Omega))^2$  such that $\bfv\in V$ and $q \in \Qo$, for every $x\coloneqq (x_1,x_2)^\top\in \Omega$ defined via 
\begin{align}
	\bfv(x)\coloneqq \vert x\vert^{\beta} (x_2,-x_1)^\top\,, \qquad q(x)\coloneqq  \vert x\vert^{\gamma}-\langle\,\vert \cdot\vert^{\gamma}\,\rangle_\Omega\,,
\end{align}
are a solutions to  \eqref{eq:p-navier-stokes}. Since, in this case, we consider given inhomogeneous Dirichlet boundary data, we do not employ Problem (\hyperlink{Qhldg}{Q$_h$})~(or~Problem~(\hyperlink{Phldg}{P$_h$}),~respectively), but their consistent inhomogeneous generalizations, \textit{i.e.}, we employ the inhomogeneous generalization of Problem (\hyperlink{Qhldg}{Q$_h$}) (or Problem (\hyperlink{Phldg}{P$_h$}), respectively) from \cite{kr-pnse-ldg-1}.

Concerning the regularity of the velocity vector field, we choose $\beta=1.0\times 10^{-2}$, which just yields that ${\bfF(\bfD\bfv)\in (W^{1,2}(\Omega))^{2\times 2}}$. Concerning the regularity of the pressure, we distinguish two~different~cases:
\begin{itemize}[{(Case 2)}]
	\item[\hypertarget{case_1}{(Case 1)}] We choose $\gamma= 1-2/p'+1.0\times 10^{-2}$, which  just  yields that $\smash{q \in
		W^{1,p'}(\Omega)}$;
	\item[\hypertarget{case_2}{(Case 2)}] We choose $\gamma= \beta(p-2)/2+1.0\times 10^{-2}$, which  just yields that $\nabla q \in
	(L^2(\Omega;\mu^{-1}))^2$.
\end{itemize}
In both Case \hyperlink{case_1}{1} and Case \hyperlink{case_2}{2}, we
have that\footnote{as to be expected in the two-dimensional case
  (\textit{cf.}\ Remark \ref{rem:muckenhoupt} (ii)).}
$\bfD\bfv\hspace*{-0.1em}\in\hspace*{-0.1em}
(C^{0,\beta}(\overline{\Omega}))^{2\times 2}$ and, thus,
$\mu_{\bfD\bfv}\hspace*{-0.1em}\coloneqq \hspace*{-0.1em}
(\delta+\vert
\overline{\bfD\bfv}\vert)^{p-2}\hspace*{-0.1em}\in\hspace*{-0.1em}
A_2(\mathbb{R}^2)$ (\textit{cf.}\ Remark \ref{rem:muckenhoupt} (i)).
Hence, in Case \hyperlink{case_1}{1}, we can expect the
	convergence rate $\texttt{EOC}_i(e_{q,i}^{\textrm{modular}})\approx p'/2$, $i=1,\ldots,6$, while in Case \hyperlink{case_2}{2},  we can expect the convergence rate ${\texttt{EOC}_i(e_{q,i}^{\textrm{modular}})\approx 1}$,~${i=1,\ldots,6}$~(\textit{cf.}\ Corollary~\ref{cor:error_pressure}).

For different values of $p\in \{2.25, 2.5, 2.75, 3, 3.25, 3.5\}$ and a
series of triangulations~$\mathcal{T}_{h_i}$, $i = 1,\ldots,6$,
obtained by regular, global refinement as described above, the EOC is
computed and presented in Table~\ref{tab1}. In both Case \hyperlink{case_1}{1} and Case \hyperlink{case_2}{2}, 
	we observe the expected convergence rate of about $\texttt{EOC}_i(e_{q,i}^{\textrm{modular}})\approx p'/2$, $i=1,\ldots,6$, (in Case \hyperlink{case_1}{1}) and ${\texttt{EOC}_i(e_{q,i}^{\textrm{modular}})\approx 1}$,~${i=1,\ldots,6}$, (in Case \hyperlink{case_2}{2}).

\begin{table}[ht]
	\setlength\tabcolsep{3pt}
	\centering
	\begin{tabular}{c |c|c|c|c|c|c|c|c|c|c|c|c|} \cmidrule(){1-13}
		\multicolumn{1}{|c||}{\cellcolor{lightgray}$\gamma$}	
		& \multicolumn{6}{c||}{\cellcolor{lightgray}Case \hyperlink{case_1}{1}}   & \multicolumn{6}{c|}{\cellcolor{lightgray}Case \hyperlink{case_2}{2}}\\ 
		\hline 
		
		\multicolumn{1}{|c||}{\cellcolor{lightgray}\diagbox[height=1.1\line,width=0.11\dimexpr\linewidth]{\vspace{-0.6mm}$i$}{\\[-5mm] $p$}}
		& \cellcolor{lightgray}2.25 & \cellcolor{lightgray}2.5  & \cellcolor{lightgray}2.75  &  \cellcolor{lightgray}3.0 & \cellcolor{lightgray}3.25  & \multicolumn{1}{c||}{\cellcolor{lightgray}3.5} &  \multicolumn{1}{c|}{\cellcolor{lightgray}2.25}   & \cellcolor{lightgray}2.5  & \cellcolor{lightgray}2.75  & \cellcolor{lightgray}3.0  & \cellcolor{lightgray}3.25 &   \cellcolor{lightgray}3.5 \\ \hline\hline
		\multicolumn{1}{|c||}{\cellcolor{lightgray}$1$}             & 0.901 & 0.795 & 0.751 & 0.720 & 0.695 & \multicolumn{1}{c||}{0.674} & \multicolumn{1}{c|}{2.160} & 1.961 & 1.834 & 1.726 & 1.676 & 1.597 \\ \hline
		\multicolumn{1}{|c||}{\cellcolor{lightgray}$2$}             & 0.900 & 0.832 & 0.784 & 0.748 & 0.720 & \multicolumn{1}{c||}{0.697} & \multicolumn{1}{c|}{1.621} & 1.531 & 1.430 & 1.396 & 1.294 & 1.245 \\ \hline
		\multicolumn{1}{|c||}{\cellcolor{lightgray}$3$}             & 0.905 & 0.839 & 0.792 & 0.756 & 0.728 & \multicolumn{1}{c||}{0.706} & \multicolumn{1}{c|}{1.160} & 1.164 & 1.167 & 1.183 & 1.242 & 1.219 \\ \hline
		\multicolumn{1}{|c||}{\cellcolor{lightgray}$4$}             & 0.907 & 0.841 & 0.793 & 0.758 & 0.730 & \multicolumn{1}{c||}{0.708} & \multicolumn{1}{c|}{1.039} & 1.043 & 1.048 & 1.073 & 1.123 & 1.125 \\ \hline
		\multicolumn{1}{|c||}{\cellcolor{lightgray}$5$}             & 0.908 & 0.841 & 0.793 & 0.758 & 0.730 & \multicolumn{1}{c||}{0.708} & \multicolumn{1}{c|}{1.016} & 1.017 & 1.019 & 1.023 & 1.044 & 1.127 \\ \hline
		\multicolumn{1}{|c||}{\cellcolor{lightgray}$6$}             & 0.908 & 0.841 & 0.793 & 0.758 & 0.730 & \multicolumn{1}{c||}{0.708} & \multicolumn{1}{c|}{1.011} & 1.012 & 1.013 & 1.014 & 1.017 & 1.060 \\ \hline\hline
		\multicolumn{1}{|c||}{\cellcolor{lightgray} theory}         & 0.900 & 0.833 & 0.786 & 0.750 & 0.722 & \multicolumn{1}{c||}{0.700} & \multicolumn{1}{c|}{1.000} & 1.000 & 1.000 & 1.000 & 1.000 & 1.000 \\ \hline
	\end{tabular}
	\caption{Experimental order of convergence (LDG): $\textup{\texttt{EOC}}_i(e_{q,i}^{\textup{modular}})$,~${i=1,\ldots,6}$.}\label{tab1}
\end{table}

\newpage
\subsubsection{\textit{(Non-)sharpness of the Muckenhoupt  regularity condition \eqref{eq:reg-assumption} in Corollary \ref{cor:error_pressure}}}

\hspace*{5mm}
In order to examine the Muckenhoupt  regularity condition \eqref{eq:reg-assumption} in  Corollary \ref{cor:error_pressure} for its sharpness, \textit{i.e.}, necessity for the convergence rates derived in Corollary \ref{cor:error_pressure}, following the construction in \cite{kr-nekorn,kr-nekorn-add}, we consider the three dimensional case
(\textit{i.e.}, $d=3$) and choose the right-hand side ${\bff\in (L^{p'}(\Omega))^3}$ such that $\bfv \in \Vo(0)$ and $q\in \Qo$, for every $x\coloneqq (x_1,x_2,x_3)^\top\in \Omega$, are defined via
\begin{align*}
	\bfv(x)\coloneqq \sum_{k\in \mathbb{N}}{\bfv^k(x)}\,,\qquad q(x)\coloneqq 25\,(\vert x\vert^\gamma -\langle \vert \cdot\vert^{\gamma} \rangle_\Omega)\,,
\end{align*}
where, for every $k\in \mathbb{N}$,
\begin{align*}
	\bfv^k(x)\coloneqq \begin{cases}
			\frac{k}{r^k}\big(\frac{r^k}{2}+\frac{\vert x-\mathbf{m}^k\vert }{2}-\vert x-\mathbf{m}^k\vert\big)(x_2-m_2^k,-(x_1-m_1^k),0)^\top&\text{ if }x\in B_{r^k}^3(\mathbf{m}^k)\,,\\
			\mathbf{0}&\text{ if }x\in \Omega\setminus B_{r^k}^3(\mathbf{m}^k)\,.
	\end{cases}
\end{align*}
for sequences $(r^k)_{k\in \mathbb{N}}\subseteq \mathbb{R}^{>0}$ and $(\mathbf{m}^k)_{k\in \mathbb{N}}\coloneqq ((m^k_1,m^k_2,m^k_3)^\top)_{k\in \mathbb{N}}\subseteq \Omega$ satisfying 
\begin{align}\label{conditions}
	r^{k+1}\leq \frac{1}{2}r^k\,,\quad B_{r^k}^3(\mathbf{m}^k)\subseteq \Omega\,,\quad B_{r^k}^3(\mathbf{m}^k)\cap B_{r^{k+1}}^3(\mathbf{m}^{k+1})=\emptyset\qquad\text{ for all }k\in \mathbb{N}\,.
\end{align} 
Following the argumentation in \cite{kr-nekorn,kr-nekorn-add}, one finds that $	\bfF(\bfD\bfv)\in (W^{1,2}(\Omega))^{3\times 3}$ 
while, on the other hand, one finds that
\begin{align}\label{violation}
	\mu_{\bfD\bfv} \coloneqq (\delta+\vert \overline{\bfD\bfv}\vert)^{p-2}\notin A_2(\mathbb{R}^3)\,.
\end{align}
In the numerical experiments, we choose $r^k\hspace*{-0.1em}\coloneqq \hspace*{-0.1em}2^{-k-2}$ and ${\mathbf{m}^k\hspace*{-0.1em}\coloneqq\hspace*{-0.1em} \mathbf{e}_1+3 r^k ( \mathbf{q}_0-\mathbf{e}_1)/\vert \mathbf{q}_0-\mathbf{e}_1\vert}$~for~all~${k\hspace*{-0.1em}\in\hspace*{-0.1em} \mathbb{N}}$, where $\mathbf{q}_0\hspace*{-0.1em}\in \hspace*{-0.1em}K_0\hspace*{-0.1em}\coloneqq\hspace*{-0.1em} \textup{conv}\{\mathbf{0},\mathbf{e}_1,\mathbf{e}_1+\mathbf{e}_2,\mathbf{e}_1+\mathbf{e}_2+\mathbf{e}_3\}\hspace*{-0.1em}\in\hspace*{-0.1em} \mathcal{T}_0$ is the quadrature point of~the~Keast~rule~\texttt{(KEAST7)} (\textit{cf.}\ \cite{keast}) in the initial triangulation $\mathcal{T}_0$ that is closest to the unit vector $\mathbf{e}_1\coloneqq (1,0,0)^\top\in \mathbb{S}^2$~(\textit{cf.} Figure~\ref{fig:1} (LEFT))\footnote{Note that there is a second quadrature point of the  Keast rule \texttt{(KEAST7)} (\textit{cf.}\ \cite{keast}) in an adjacent element that has the same distance to the first unit vector $\mathbf{e}_1\in \mathbb{S}^2$.}. In this way, we can guarantee that conditions \eqref{conditions} are met as well as that for every $i\in \mathbb{N}$, there exists an integer $N_i\in \mathbb{N}$ such that 
\begin{align*}
	\mathbf{q}_i\coloneqq (1-2^{-N_i})\,\mathbf{e}_1+2^{-N_i}\, \mathbf{q}_0\in \overline{B_{r^{N_i}}^3(\mathbf{m}^{N_i})}\,,
\end{align*}
which is the quadrature point of the  Keast rule \texttt{(KEAST7)} (\textit{cf.}\ \cite{keast}) in $\mathcal{T}_i$ that is closest to the unit vector $\mathbf{e}_1\in \mathbb{S}^2$. 
This has the consequence that for every $k\in \mathbb{N}$ with $k\ge N_i+1$, the ball $B_{r^k}^3(\mathbf{m}^k)$ does not contain any quadrature point of the Keast rule \texttt{(KEAST7)} (\textit{cf.}\ \cite{keast}) in $\mathcal{T}_i$. Hence,~in~the~numerical~experiments, for every refinement step $i\in \mathbb{N}$,~we~can~interchange $	\bfv\coloneqq \sum_{k\in \mathbb{N}}{\bfv^k}\in \smash{(W^{1,p}_0(\Omega))^3}$ (which cannot be directly implemented) by 	$\overline{\bfv}^{\smash{N_i}}\coloneqq \sum_{k=1}^{\smash{N_i}}{\bfv^k}\in \smash{(W^{1,p}_0(\Omega))^3}$ (\textit{cf.}\ Figure \ref{fig:1} and Figure \ref{fig:2}(MIDDLE/RIGHT))
(which can directly be implemented); therefore,~we~can interchange  $\bff\coloneqq -\textup{div}\,\bfS(\bfD\bfv)+[\nabla\bfv]\bfv+\nabla q\in (L^{p'}(\Omega))^3$ by $\overline{\bff}^{\smash{N_i}}\coloneqq  -\textup{div}\,\bfS(\bfD\bfv^{\smash{N_i}})+[\nabla\bfv^{\smash{N_i}}]\bfv^{\smash{N_i}}+\nabla q\in (L^{p'}(\Omega))^3$ without changing the discrete formulations.  
Also note that the violation of the Muckenhoupt condition \eqref{violation} can be asymptotically confirmed 
using the Keast rule \texttt{(KEAST7)} (\textit{cf.}\ \cite{keast}). In fact, in Figure \ref{fig:3}, denoting by $\mathrm{d}\mu_{h_i}$ the discrete measure representing the Keast rule \texttt{(KEAST7)} (\textit{cf.}\ \cite{keast}) in $\mathcal{T}_i$, it is indicated that
\begin{align*}
	E_i\coloneqq \bigg(\int_{B_{r^{N_i}}^3(\mathbf{m}^{N_i})}{\mu_{\bfD\bfv} \,\mathrm{d}\mu_{h_i}}\bigg)\bigg(\int_{B_{r^{N_i}}^3(\mathbf{m}^{N_i})}{\mu_{\bfD\bfv} ^{-1}\,\mathrm{d}\mu_{h_i}}\bigg)\to\infty \quad(i\to \infty)\,,
\end{align*}
indicating that the violation of the Muckenhoupt condition \eqref{violation} is sufficiently resolved by the Keast rule \texttt{(KEAST7)} (\textit{cf.}\ \cite{keast}).
\newpage

\begin{figure}[ht]
	\hspace*{-0.25cm}\includegraphics[width=7.5cm]{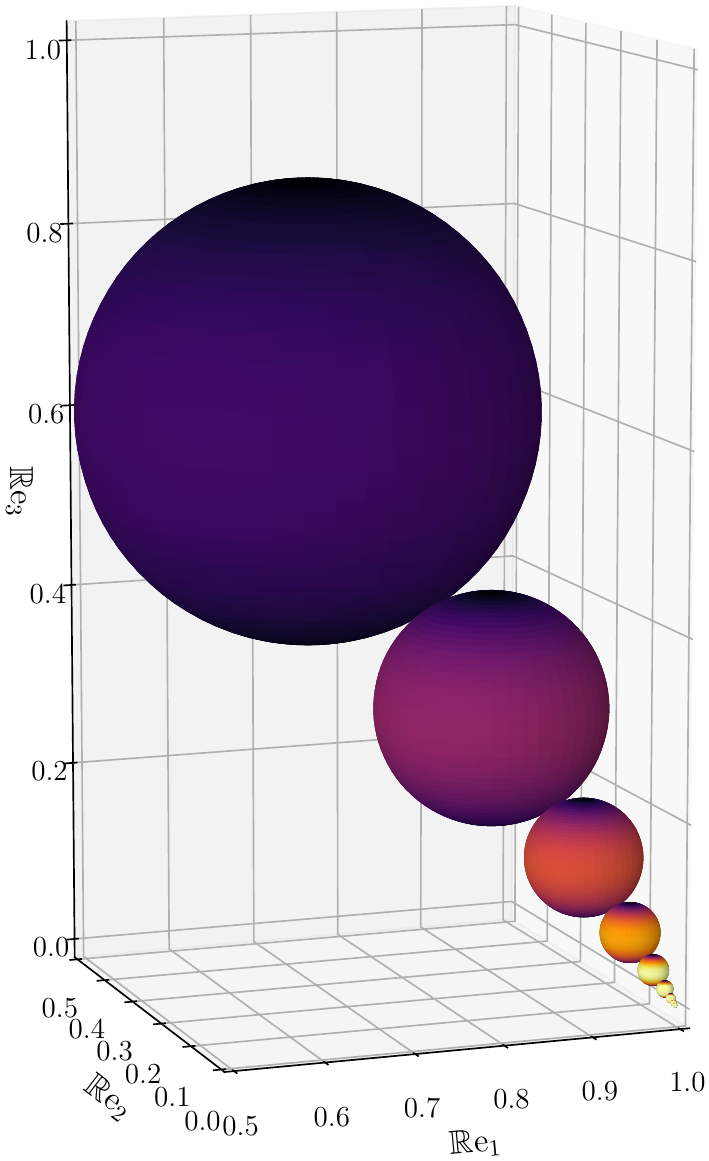}\includegraphics[width=7.5cm]{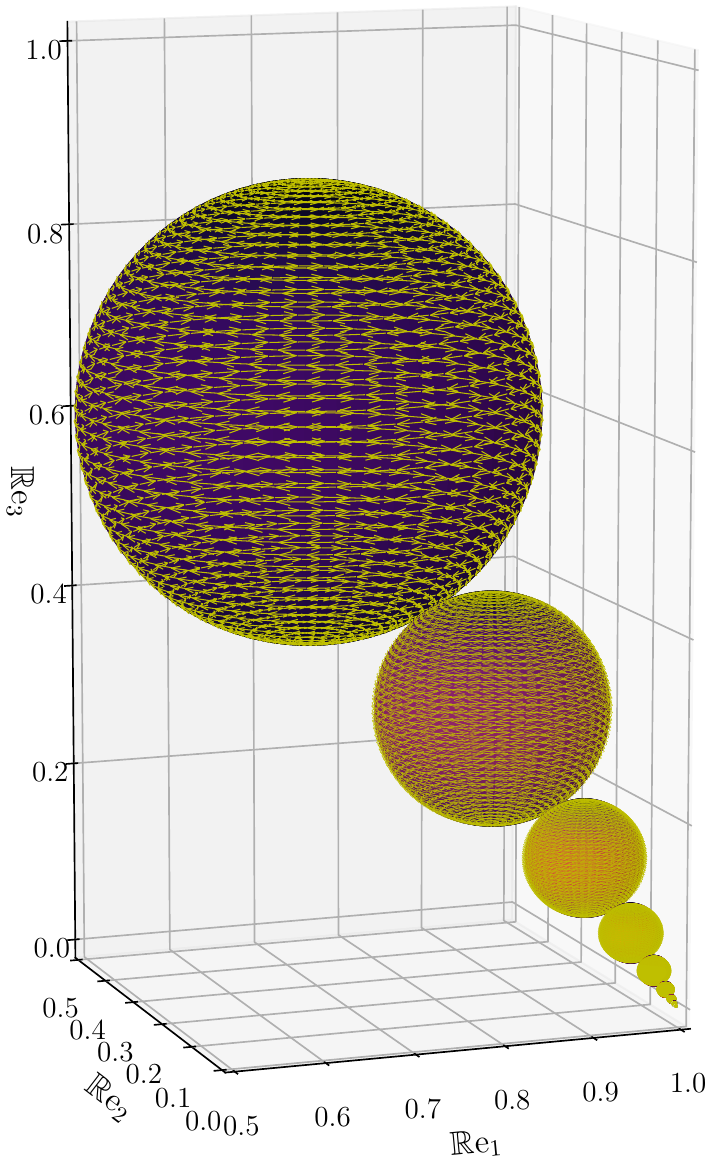}
	\caption{LEFT: Plot of  the strain rate $\vert\bfD \bfv\vert\colon \Omega\to \mathbb{R}_{\ge 0}$ restricted to the boundary of its support, \textit{i.e.},
		$\partial(\textup{supp}(\vert\bfD \bfv\vert))=\bigcup_{k=1}^\infty{\partial B_{r^k}^3(\mathbf{m}^k)}$.  The color map indicates that the strain rate 
	increases when approaching the first unit vector $\mathbf{e}_1\in \mathbb{S}^2$. RIGHT: Plot  of the velocity vector field $\bfv\colon \Omega\to \mathbb{R}^3$ restricted to the boundary of its support, \textit{i.e.},
	$\partial(\textup{supp}(\vert\bfv\vert))=\bigcup_{k=1}^\infty{\partial B_{r^k}^3(\mathbf{m}^k)}$.}\label{fig:1}
\end{figure}\vspace*{-10mm}\enlargethispage{10mm}

\begin{figure}[H]
	\hspace*{-0.25cm}\includegraphics[width=5cm]{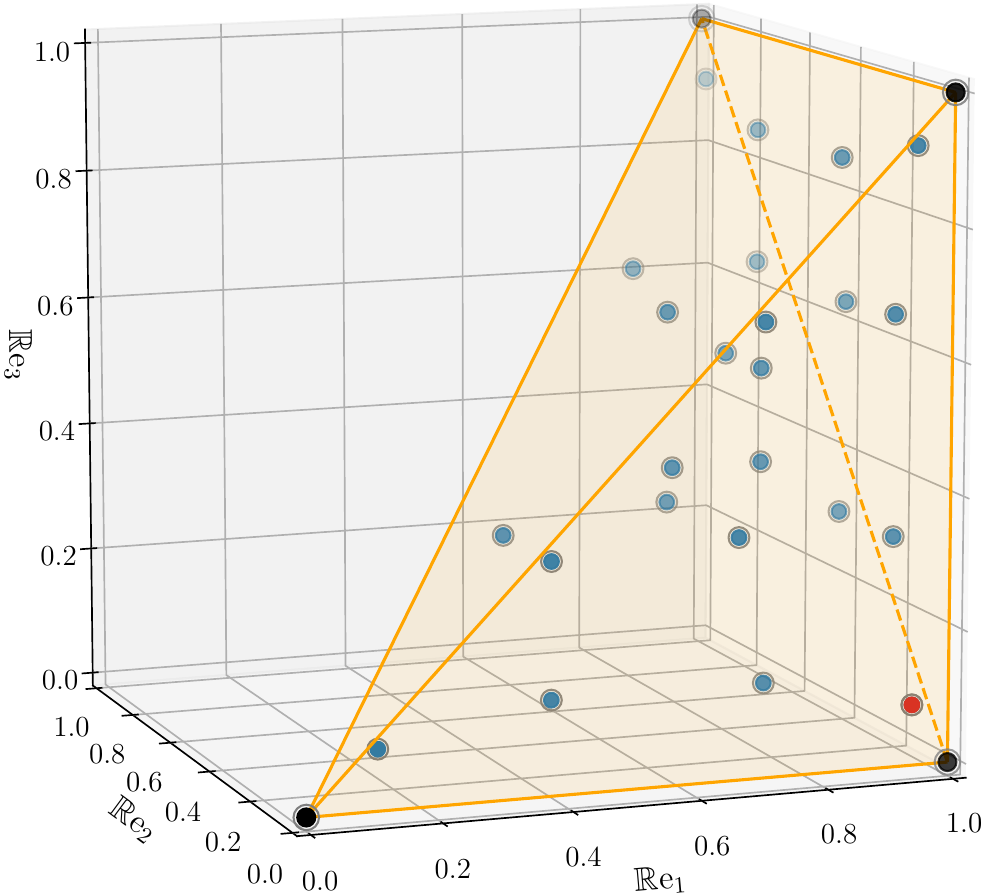}\includegraphics[width=5cm]{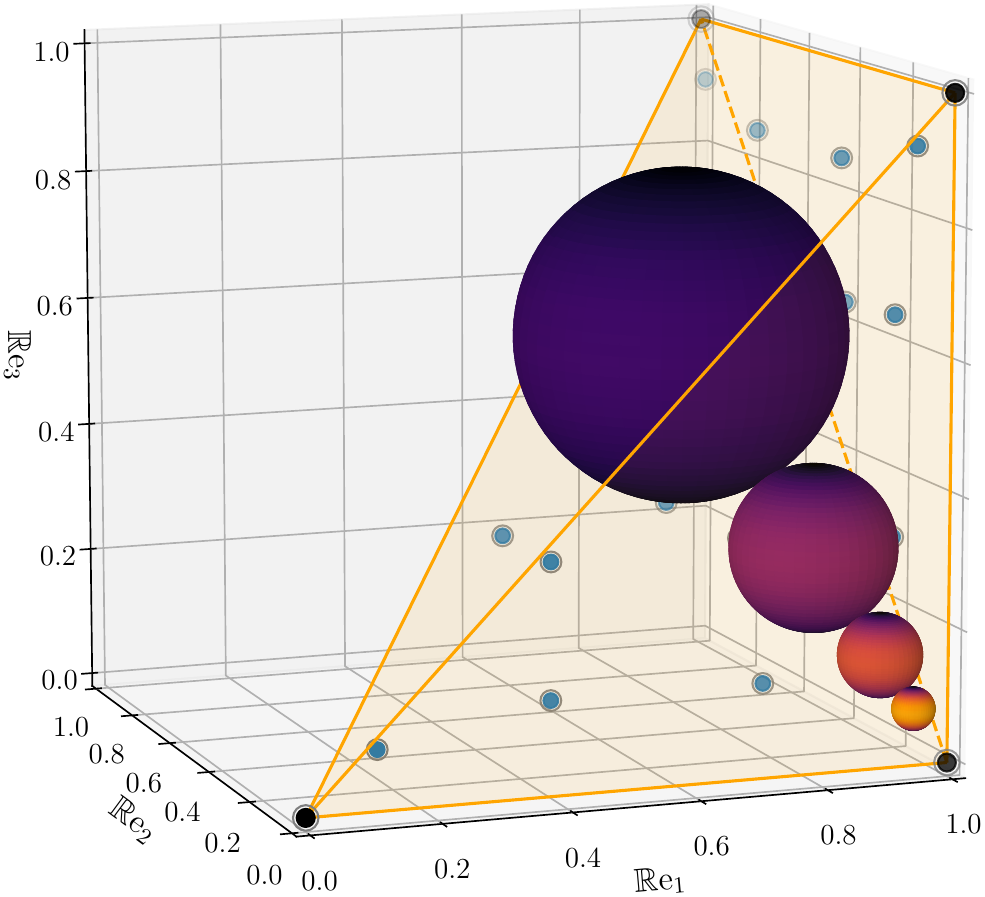}\includegraphics[width=5cm]{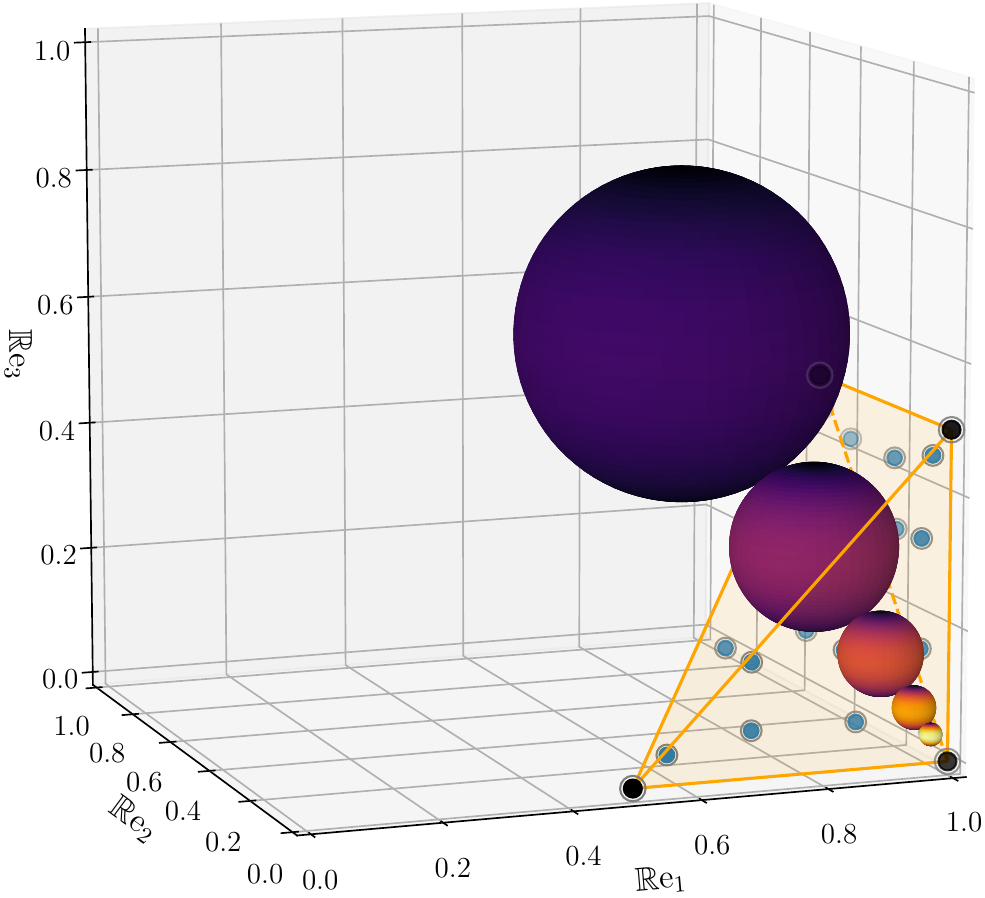}
	\caption{LEFT: Plot of the tetrahedron $K_0\coloneqq \textup{conv}\{\mathbf{0},\mathbf{e}_1,\mathbf{e}_1+\mathbf{e}_2,\mathbf{e}_1+\mathbf{e}_2+\mathbf{e}_3\}$ (in orange) and the corresponding transformed 24 quadrature points (in blue and red) of the  Keast rule \texttt{(KEAST7)} (\textit{cf.}\ \cite{keast}). The quadrature point $\mathbf{q}_0\in K_0$ closest to the first unit vector is marked in red; MIDDLE: Plot of $\overline{\bfv}^{N_0}\coloneqq \sum_{k=1}^{\smash{N_0}}{\bfv^k}\in \smash{(W^{1,p}_0(\Omega))^3}$, where $N_0=4$ is minimal such that
		$\mathbf{q}_0\in \textup{supp}(\overline{\bfv}^{N_0})$; RIGHT:
		Plot of $\overline{\bfv}^{N_1}\coloneqq \sum_{k=1}^{\smash{N_1}}{\bfv^k}\in \smash{(W^{1,p}_0(\Omega))^3}$, where $N_1=5$ is minimal such that
		$\mathbf{q}_1\in \textup{supp}(\overline{\bfv}^{N_1})$.}\label{fig:2}
\end{figure}\newpage

For different values of $p\in \{2.25, 2.5, 2.75, 3, 3.25, 3.5\}$ and a
series of triangulations~$\mathcal{T}_{h_i}$, $i = 1,2,3$,
obtained \hspace*{-0.1mm}by \hspace*{-0.1mm}regular, \hspace*{-0.1mm}global \hspace*{-0.1mm}refinement \hspace*{-0.1mm}as \hspace*{-0.1mm}described \hspace*{-0.1mm}above, \hspace*{-0.1mm}the \hspace*{-0.1mm}EOC \hspace*{-0.1mm}is~\hspace*{-0.1mm}computed~\hspace*{-0.1mm}and~\hspace*{-0.1mm}presented~\hspace*{-0.1mm}in~\hspace*{-0.1mm}Tables~\hspace*{-0.1mm}\mbox{\ref{tab2}--\ref{tab5}}. In both Case \hyperlink{case_1}{1} and Case \hyperlink{case_2}{2}, 
we observe the expected convergence rate of about $\texttt{EOC}_i(e_{\bfF,i})\approx\texttt{EOC}_i(e_{\jump{},i})\approx p'/2$, $i=1,2,3$,   (in Case \hyperlink{case_1}{1}) and $\texttt{EOC}_i(e_{\bfF,i})\approx\texttt{EOC}_i(e_{\jump{},i})\approx 1$, $ i=1,2,3$, (in Case \hyperlink{case_2}{2}). On the other hand, in both Case \hyperlink{case_1}{1} and Case \hyperlink{case_2}{2}, although the Muckenhoupt regularity~condition~\eqref{eq:reg-assumption}~is~not~met~(\textit{cf}.~\eqref{violation}), we 
report the convergence rate of about $\texttt{EOC}_i(e_{q,i}^{\textup{modular}})\approx p'/2$, $i=1,2,3$,   (in Case \hyperlink{case_1}{1}) and $\texttt{EOC}_i(e_{q,i}^{\textup{modular}})\approx 1$, $ i=1,2,3$, (in Case \hyperlink{case_2}{2}). 
This shows that the Muckenhoupt regularity condition \eqref{eq:reg-assumption} is merely sufficient, but not necessary, for the convergence rates derived in Corollary \ref{cor:error_pressure}. In addition,~in~both~Case~\hyperlink{case_1}{1}~and~Case~\hyperlink{case_2}{2}, we observe the increased convergence rate of about $\texttt{EOC}_i(e_{q,i}^{\textup{norm}})\approx 1$, $i=1,2,3$,   (in Case \hyperlink{case_1}{1}) and $\texttt{EOC}_i(e_{q,i}^{\textup{norm}})\approx 2/p'$, $ i=1,2,3$, (in Case \hyperlink{case_2}{2}). This indicates that the convergence rates for the kinematic pressure derived in Corollary \ref{cor:error_LDG} are also potentially sub-optimal in three dimensions.\enlargethispage{5mm}

\begin{table}[ht]
	\setlength\tabcolsep{3pt}
	\centering
	\begin{tabular}{c |c|c|c|c|c|c|c|c|c|c|c|c|} \cmidrule(){1-13}
		\multicolumn{1}{|c||}{\cellcolor{lightgray}$\gamma$}	
		& \multicolumn{6}{c||}{\cellcolor{lightgray}Case \hyperlink{case_1}{1}}   & \multicolumn{6}{c|}{\cellcolor{lightgray}Case \hyperlink{case_2}{2}}\\ 
		\hline 
		
		\multicolumn{1}{|c||}{\cellcolor{lightgray}\diagbox[height=1.1\line,width=0.11\dimexpr\linewidth]{\vspace{-0.6mm}$i$}{\\[-5mm] $p$}}
		& \cellcolor{lightgray}2.25 & \cellcolor{lightgray}2.5  & \cellcolor{lightgray}2.75  &  \cellcolor{lightgray}3.0 & \cellcolor{lightgray}3.25  & \multicolumn{1}{c||}{\cellcolor{lightgray}3.5} &  \multicolumn{1}{c|}{\cellcolor{lightgray}2.25}   & \cellcolor{lightgray}2.5  & \cellcolor{lightgray}2.75  & \cellcolor{lightgray}3.0  & \cellcolor{lightgray}3.25 &   \cellcolor{lightgray}3.5 \\ \hline\hline
		\multicolumn{1}{|c||}{\cellcolor{lightgray}$1$}             & 0.832 & 0.794 & 0.759 & 0.719 & 0.685 & \multicolumn{1}{c||}{0.648} & \multicolumn{1}{c|}{0.947} & 0.977 & 0.983 & 0.970 & 0.956 & 0.934 \\ \hline
		\multicolumn{1}{|c||}{\cellcolor{lightgray}$2$}             & 0.885 & 0.817 & 0.768 & 0.731 & 0.702 & \multicolumn{1}{c||}{0.679} & \multicolumn{1}{c|}{1.017} & 1.026 & 1.027 & 1.024 & 1.019 & 1.013 \\ \hline
		\multicolumn{1}{|c||}{\cellcolor{lightgray}$3$}             & 0.902 & 0.834 & 0.785 & 0.748 & 0.720 & \multicolumn{1}{c||}{0.697} & \multicolumn{1}{c|}{1.043} & 1.063 & 1.072 & 1.074 & 1.073 & 1.070 \\ \hline
		\multicolumn{1}{|c||}{\cellcolor{lightgray} theory}         & 0.900 & 0.833 & 0.786 & 0.750 & 0.722 & \multicolumn{1}{c||}{0.700} & \multicolumn{1}{c|}{1.000} & 1.000 & 1.000 & 1.000 & 1.000 & 1.000 \\ \hline
	\end{tabular}
	\caption{Experimental order of convergence (LDG): $\texttt{EOC}_i(e_{\bfF,i})$,~${i=1,2,3}$.}\vspace*{1mm}
	\label{tab2}
	\setlength\tabcolsep{3pt}
	\centering
	\begin{tabular}{c |c|c|c|c|c|c|c|c|c|c|c|c|} \cmidrule(){1-13}
		\multicolumn{1}{|c||}{\cellcolor{lightgray}$\gamma$}	
		& \multicolumn{6}{c||}{\cellcolor{lightgray}Case \hyperlink{case_1}{1}}   & \multicolumn{6}{c|}{\cellcolor{lightgray}Case \hyperlink{case_2}{2}}\\ 
		\hline 
		
		\multicolumn{1}{|c||}{\cellcolor{lightgray}\diagbox[height=1.1\line,width=0.11\dimexpr\linewidth]{\vspace{-0.6mm}$i$}{\\[-5mm] $p$}}
		& \cellcolor{lightgray}2.25 & \cellcolor{lightgray}2.5  & \cellcolor{lightgray}2.75  &  \cellcolor{lightgray}3.0 & \cellcolor{lightgray}3.25  & \multicolumn{1}{c||}{\cellcolor{lightgray}3.5} &  \multicolumn{1}{c|}{\cellcolor{lightgray}2.25}   & \cellcolor{lightgray}2.5  & \cellcolor{lightgray}2.75  & \cellcolor{lightgray}3.0  & \cellcolor{lightgray}3.25 &   \cellcolor{lightgray}3.5 \\ \hline\hline
		\multicolumn{1}{|c||}{\cellcolor{lightgray}$1$}             & 0.761 & 0.549 & 0.325 & 0.082 & -0.17 & \multicolumn{1}{c||}{-0.43} & \multicolumn{1}{c|}{0.884} & 0.770 & 0.647 & 0.526 & 0.428 & 0.348 \\ \hline
		\multicolumn{1}{|c||}{\cellcolor{lightgray}$2$}             & 0.886 & 0.818 & 0.769 & 0.733 & 0.705 & \multicolumn{1}{c||}{0.684} & \multicolumn{1}{c|}{1.019} & 1.030 & 1.032 & 1.029 & 1.024 & 1.018 \\ \hline
		\multicolumn{1}{|c||}{\cellcolor{lightgray}$3$}             & 0.903 & 0.834 & 0.786 & 0.749 & 0.721 & \multicolumn{1}{c||}{0.699} & \multicolumn{1}{c|}{1.045} & 1.066 & 1.076 & 1.080 & 1.079 & 1.076 \\ \hline
		\multicolumn{1}{|c||}{\cellcolor{lightgray} theory}         & 0.900 & 0.833 & 0.786 & 0.750 & 0.722 & \multicolumn{1}{c||}{0.700} & \multicolumn{1}{c|}{1.000} & 1.000 & 1.000 & 1.000 & 1.000 & 1.000 \\ \hline
	\end{tabular}
	\caption{Experimental order of convergence (LDG): $\texttt{EOC}_i(e_{\jump{},i})$,~${i=1,2,3}$.}\vspace*{1mm}
	\label{tab3}
	\setlength\tabcolsep{3pt}
	\centering
	\begin{tabular}{c |c|c|c|c|c|c|c|c|c|c|c|c|} \cmidrule(){1-13}
		\multicolumn{1}{|c||}{\cellcolor{lightgray}$\gamma$}	
		& \multicolumn{6}{c||}{\cellcolor{lightgray}Case \hyperlink{case_1}{1}}   & \multicolumn{6}{c|}{\cellcolor{lightgray}Case \hyperlink{case_2}{2}}\\ 
		\hline 
		
		\multicolumn{1}{|c||}{\cellcolor{lightgray}\diagbox[height=1.1\line,width=0.11\dimexpr\linewidth]{\vspace{-0.6mm}$i$}{\\[-5mm] $p$}}
		& \cellcolor{lightgray}2.25 & \cellcolor{lightgray}2.5  & \cellcolor{lightgray}2.75  &  \cellcolor{lightgray}3.0 & \cellcolor{lightgray}3.25  & \multicolumn{1}{c||}{\cellcolor{lightgray}3.5} &  \multicolumn{1}{c|}{\cellcolor{lightgray}2.25}   & \cellcolor{lightgray}2.5  & \cellcolor{lightgray}2.75  & \cellcolor{lightgray}3.0  & \cellcolor{lightgray}3.25 &   \cellcolor{lightgray}3.5 \\ \hline\hline
		\multicolumn{1}{|c||}{\cellcolor{lightgray}$1$}             & 0.899 & 0.918 & 0.941 & 0.970 & 0.987 & \multicolumn{1}{c||}{1.008} & \multicolumn{1}{c|}{1.019} & 1.105 & 1.164 & 1.209 & 1.236 & 1.260 \\ \hline
		\multicolumn{1}{|c||}{\cellcolor{lightgray}$2$}             & 0.986 & 0.979 & 0.974 & 0.971 & 0.969 & \multicolumn{1}{c||}{0.967} & \multicolumn{1}{c|}{1.137} & 1.237 & 1.309 & 1.363 & 1.404 & 1.436 \\ \hline
		\multicolumn{1}{|c||}{\cellcolor{lightgray}$3$}             & 1.004 & 1.001 & 1.000 & 0.999 & 1.000 & \multicolumn{1}{c||}{1.001} & \multicolumn{1}{c|}{1.163} & 1.280 & 1.366 & 1.432 & 1.482 & 1.522 \\ \hline
		\multicolumn{1}{|c||}{\cellcolor{lightgray} theory}         & 0.900 & 0.833 & 0.786 & 0.750 & 0.722 & \multicolumn{1}{c||}{0.700} & \multicolumn{1}{c|}{1.000} & 1.000 & 1.000 & 1.000 & 1.000 & 1.000 \\ \hline
	\end{tabular}
	\caption{Experimental order of convergence (LDG): $\texttt{EOC}_i(e_{q,i}^{\textup{norm}})$,~${i=1,2,3}$.}\vspace*{1mm}
	\label{tab4}
	\setlength\tabcolsep{3pt}
	\centering
	\begin{tabular}{c |c|c|c|c|c|c|c|c|c|c|c|c|} \cmidrule(){1-13}
		\multicolumn{1}{|c||}{\cellcolor{lightgray}$\gamma$}	
		& \multicolumn{6}{c||}{\cellcolor{lightgray}Case \hyperlink{case_1}{1}}   & \multicolumn{6}{c|}{\cellcolor{lightgray}Case \hyperlink{case_2}{2}}\\ 
		\hline 
		
		\multicolumn{1}{|c||}{\cellcolor{lightgray}\diagbox[height=1.1\line,width=0.11\dimexpr\linewidth]{\vspace{-0.6mm}$i$}{\\[-5mm] $p$}}
		& \cellcolor{lightgray}2.25 & \cellcolor{lightgray}2.5  & \cellcolor{lightgray}2.75  &  \cellcolor{lightgray}3.0 & \cellcolor{lightgray}3.25  & \multicolumn{1}{c||}{\cellcolor{lightgray}3.5} &  \multicolumn{1}{c|}{\cellcolor{lightgray}2.25}   & \cellcolor{lightgray}2.5  & \cellcolor{lightgray}2.75  & \cellcolor{lightgray}3.0  & \cellcolor{lightgray}3.25 &   \cellcolor{lightgray}3.5 \\ \hline\hline
		\multicolumn{1}{|c||}{\cellcolor{lightgray}$1$}             & 0.809 & 0.765 & 0.740 & 0.728 & 0.713 & \multicolumn{1}{c||}{0.706} & \multicolumn{1}{c|}{0.917} & 0.921 & 0.915 & 0.907 & 0.893 & 0.882 \\ \hline
		\multicolumn{1}{|c||}{\cellcolor{lightgray}$2$}             & 0.888 & 0.816 & 0.765 & 0.728 & 0.700 & \multicolumn{1}{c||}{0.677} & \multicolumn{1}{c|}{1.024} & 1.031 & 1.029 & 1.022 & 1.014 & 1.006 \\ \hline
		\multicolumn{1}{|c||}{\cellcolor{lightgray}$3$}             & 0.903 & 0.834 & 0.785 & 0.749 & 0.722 & \multicolumn{1}{c||}{0.701} & \multicolumn{1}{c|}{1.047} & 1.066 & 1.074 & 1.074 & 1.071 & 1.066 \\ \hline
		\multicolumn{1}{|c||}{\cellcolor{lightgray} theory}         & 0.900 & 0.833 & 0.786 & 0.750 & 0.722 & \multicolumn{1}{c||}{0.700} & \multicolumn{1}{c|}{1.000} & 1.000 & 1.000 & 1.000 & 1.000 & 1.000 \\ \hline
	\end{tabular}
	\caption{Experimental order of convergence (LDG): $\texttt{EOC}_i(e_{q,i}^{\textup{modular}})$,~${i=1,2,3}$.}
	\label{tab5}
\end{table}

\begin{figure}[H]
	\includegraphics[width=14.5cm]{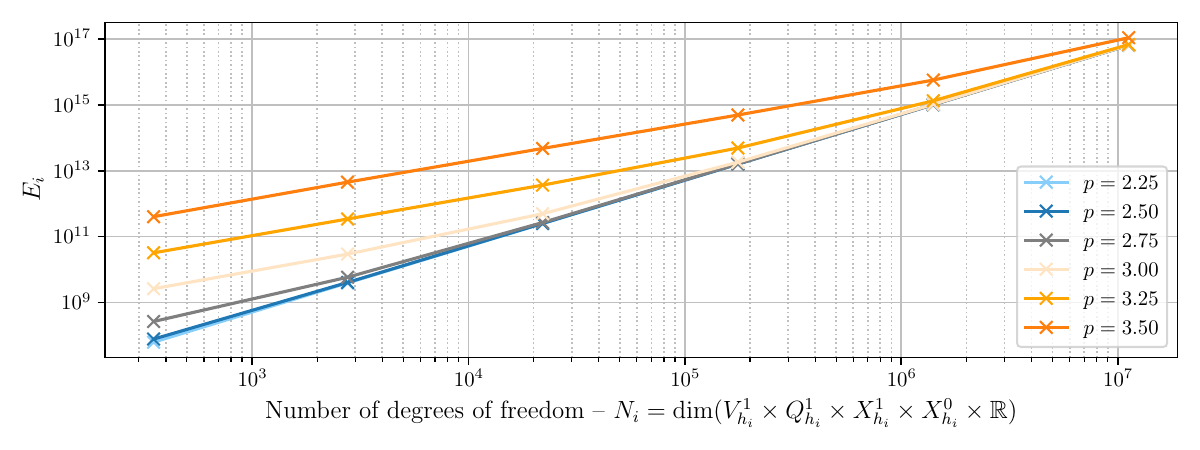}\vspace*{-2mm}
	\caption{Plots of $E_i\coloneqq (\int_{B_{r^{N_i}}^3(\mathbf{m}^{N_i})}{\mu_{\bfD\bfv} \,\mathrm{d}\mu_{h_i}})(\int_{B_{r^{N_i}}^3(\mathbf{m}^{N_i})}{\mu_{\bfD\bfv} ^{-1}\,\mathrm{d}\mu_{h_i}})$, $i=1,\ldots,6$, for $p\in \{2.25,2.5,2.75,3.0,3.25,2.5\}$, where $\mathrm{d}\mu_{h_i}$, $i=1,\ldots,6$, denote the discrete measures representing the Keast rule \texttt{(KEAST7)} (\textit{cf.}\ \cite{keast}), indicating  that $E_i\to \infty$ $(i\to \infty)$~and,~thus, that the violation of the Muckenhoupt condition \eqref{violation} is sufficiently resolved by the Keast rule \texttt{(KEAST7)} (\textit{cf.}\ \cite{keast}).}\label{fig:3}
\end{figure}
 
\appendix
\section{Appendix}

\hspace*{5mm}In this appendix, we give a proof of a local stability and approximation result for a locally $L^1$-stable projection operator in terms of Muckenhoupt weights.

\begin{lemma}\label{lem:PiYstab_muckenhoupt}
	Let  $p\in (1,\infty)$ and 
	 $\sigma \in A_{p'}(\mathbb{R}^d)$. 
	 	If 	$\Pi_h^Q\colon Q \to \Qhk$ is a linear projection operator, that is $\Pi_h^Qz_h=z_h$ for all $z_h\in \Qhk$, which is \textup{locally $L^1$-stable} (\textit{i.e.}, \eqref{eq:PiYstab}), then there exists a constant $c>0$, depending on $k$, $\omega_0$, $p$, $\Omega$, and $[\sigma]_{A_{p'}(\mathbb{R}^d)}$, such that for every $z\in Q$ and $K\in \mathcal{T}_h$, it holds that
	\begin{align}
				\|\Pi_h^Q z\|_{p',\sigma,K}&\leq c\, \|z\|_{p',\sigma,\omega_K}\,,\label{lem:PiYstab_muckenhoupt.1}\\
			\|z-\Pi_h^Q z\|_{p',\sigma,K}&\leq c\, h_K\,\|\nabla z\|_{p',\sigma,\omega_K}\,.\label{lem:PiYstab_muckenhoupt.2}
	\end{align}
\end{lemma}

\begin{proof}
	\textit{ad \eqref{lem:PiYstab_muckenhoupt.1}.}
	Using H\"older's inequality, a local  norm equivalence  (\textit{cf.}\ \cite[Lem.\ 12.1]{EG21}), the local $L^1$-stability of $\Pi_h^Q$ (\textit{cf.}\ \eqref{eq:PiYstab}), $\sigma\in
	A_{p'}(\mathbb{R}^d)$, and   $\vert
	B_K\vert\sim \vert K\vert\sim\vert \omega_K\vert$, where $B_K\coloneqq B_{\textup{diam}(\omega_K)}^d(x_K)$ and $x_K$ is the barycenter of $K$, we find that
	\begin{align*}
			\|\Pi_h^Q z\|_{p',\sigma,K}&\leq c\, \|\Pi_h^Qz\|_{\infty,K}\smash{\|\sigma\|_{1,K}^{1/\smash{p'}}}\\&\leq 
			c\,\vert K\vert^{-1}\smash{\| \Pi_h^Q z\|_{1,K} \|\sigma\|_{1,K}^{1/\smash{p'}}}
			\\&\leq c\,	\vert K\vert^{-1}\smash{\| z\|_{1,\omega_K} \|\sigma\|_{1,K}^{1/\smash{p'}}}
			\\&\leq c\,	\vert K\vert^{-1}\smash{\|z\sigma^{1/\smash{p'}}\sigma^{-1/p'}\|_{1,\omega_K} \|\sigma\|_{1,K}^{1/\smash{p'}}}
			\\&\leq c\, \|z\|_{p',\sigma,\omega_K}\smash{\|\vert \omega_K\vert^{-1}\sigma^{1-p}\|_{1,\omega_K}^{1/p} \|\vert K\vert^{-1}\sigma\|_{1,K}^{1/\smash{p'}}}
			\\&\leq c\,	\|z\|_{p',\sigma,\omega_K}\smash{\|\vert B_K\vert^{-1}\sigma^{1-p}\|_{1,B_K}^{1/p} \|\vert B_K\vert^{-1}\sigma\|_{1,B_K}^{1/\smash{p'}}}
			\\&\leq c\, \|z\|_{p',\sigma,\omega_K}\,,
	\end{align*}
	which is the claimed local stability estimate \eqref{lem:PiYstab_muckenhoupt.1}.
	
	\textit{ad \eqref{lem:PiYstab_muckenhoupt.2}.}
     From \cite[Lem.~8.2.1, Rem.~8.2.10,
        Lem.~6.1.4]{lpx-book}, for a.e.\ $x \in \omega_K$, it follows that
	\begin{align*}
		\vert z(x)-\langle z\rangle_{\omega_K}\vert&\leq c\,
          \int_{\mathbb R^d} \frac{\abs{\nabla
          z(y)}\chi_{\omega_K}(y)}{\abs{x-y}^{d-1}}\, \mathrm{d}x \\&\leq c\,
          \int_{B_{\textup{diam}(\omega_K)}^d(x)} \frac{\abs{\nabla
          z(y)}\chi_{\omega_K}(y)}{\abs{x-y}^{d-1}}\, \mathrm{d}x \\&\le c\, h_K\,M_d(\nabla z\chi_{\omega_K})(x)\,,
	\end{align*}
        where we also used that $h_K\sim \textup{diam}(\omega_K)$. 
	Using this estimate, that $\Pi_h^Q\langle z\rangle_{\omega_K}=\langle z\rangle_{\omega_K}$, the local stability estimate \eqref{lem:PiYstab_muckenhoupt.1}, and that $M_d$ is stable from $L^{p'}(\mathbb{R}^d;\sigma)$ to $L^{p'}(\mathbb{R}^d;\sigma)$, we find that
	\begin{align*}\label{lem:shifted_modular_Pi_div.1}
          \begin{aligned}
            \|z-\Pi_h^Q z\|_{p',\sigma,K}&\leq \smash{\|z-\langle z\rangle_{\omega_K}\|_{p',\sigma,K}+ \|\Pi_h^Q(z-\langle z\rangle_{\omega_K})\|_{p',\sigma,K}}
            \\
            &\leq \smash{c\,\|z-\langle z\rangle_{\omega_K}\|_{p',\sigma,\omega_K}}
            \\
            &\leq \smash{c\,h_K\,\|M_d(\nabla z\chi_{\omega_K})\|_{p',\sigma,\omega_K}}
              \\
              &\leq \smash{c\,h_K\,\|M_d(\nabla z\chi_{\omega_K})\|_{p',\sigma,\mathbb{R}^d}}
              \\
              &\leq \smash{c\,h_K\,\|\nabla z\|_{p',\sigma,\omega_K}\,,}
		\end{aligned}
	\end{align*}
	which is the claimed local approximation estimate \eqref{lem:PiYstab_muckenhoupt.2}.
\end{proof}
%


\def\cprime{$'$} \def\cprime{$'$} \def\cprime{$'$}

\end{document}